\documentclass[11pt]{scrarticle}
\usepackage[english]{babel}
\usepackage[letterpaper,top=2.54cm,bottom=2.54cm,left=2.54cm,right=2.54cm,marginparwidth=1.75cm]{geometry}

\usepackage{amsmath,amssymb,amsfonts}
\usepackage{amsthm}
\usepackage{graphicx}
\usepackage{subcaption}
\usepackage[colorlinks=true, allcolors=blue]{hyperref}
\usepackage{lscape}
\usepackage{hhline}
\usepackage{natbib}

\usepackage{tikz}
\usetikzlibrary{shapes.geometric, arrows}
\usetikzlibrary{positioning}
\usepackage{multirow}
\usepackage{algorithm}
\usepackage[indLines=false]{algpseudocodex}
\usepackage{float}
\usepackage[onehalfspacing]{setspace}

\usepackage{authblk}

\newtheorem{theorem}{Theorem}

\newtheorem{definition}{Definition}
\newtheorem{lemma}{Lemma}

\newtheorem{example}{Example}

\newcommand{\Z}{\mathbb{Z}}

\newcommand{\cart}{\texttt{CART}}
\newcommand{\cartsp}{\texttt{CART} }
\newcommand{\cartpure}{\texttt{CARTpure}}
\newcommand{\cartpuresp}{\texttt{CARTpure} }
\newcommand{\miqp}{\texttt{MIQP}}
\newcommand{\miqpsp}{\texttt{MIQP} }
\newcommand{\samx}{\texttt{CE-DT}}
\newcommand{\samxsp}{\texttt{CE-DT} }
\newcommand{\exact}{\texttt{EXACT}}
\newcommand{\exactsp}{\texttt{EXACT} }
\newcommand{\LSone}{\texttt{LS1}}
\newcommand{\LSonesp}{\texttt{LS1} }
\newcommand{\LStwo}{\texttt{LS2}}
\newcommand{\LStwosp}{\texttt{LS2} }
\newcommand{\LSthree}{\texttt{LS3}}
\newcommand{\LSthreesp}{\texttt{LS3} }
\newcommand{\REpre}{\texttt{REpre}}
\newcommand{\REpresp}{\texttt{REpre} }
\newcommand{\REpost}{\texttt{REpost}}
\newcommand{\REpostsp}{\texttt{REpost} }

\definecolor{CustomLightGray}{HTML}{ededed}
\definecolor{gred}{HTML}{BB5566}
\definecolor{ggreen}{HTML}{117733}

\DeclareMathOperator*{\argmax}{arg\,max}
\DeclareMathOperator*{\argmin}{arg\,min}

\DeclareMathOperator*{\cl}{cl}

\newcommand{\RepeatTimes}[1]{\While{$#1$ \textbf{times:}}}
\newcommand{\EndRepeat}{\EndWhile}
\algtext{While}{\textbf{repeat }}

\title{Decision-Tree Surrogates for Weak Counterfactual Explanations in Integer Optimization}

\date{}

\author[1]{Marc Goerigk}

\author[2]{Jannis Kurtz}

\author[1]{Sebastian Merten\thanks{Corresponding author. Email: sebastian.merten@uni-passau.de}}

\affil[1]{Business Decisions and Data Science, University of Passau, Passau, Germany}

\affil[2]{Amsterdam Business School, University of Amsterdam, Amsterdam, Netherlands}

\begin{document}
\maketitle

\begin{abstract}
The concept of counterfactual explanations (CEs) has recently gained increasing attention for deriving explanations of mathematical optimization problems. For a given problem, a counterfactual is an ideally minimal change in the mutable problem parameters such that an optimal solution of the changed problem exists fulfilling some desired criterion. It was shown that so-called weak CEs can be calculated by solving a computationally demanding bilevel optimization problem, which is $\Sigma_2^p$-complete if decision variables and mutable parameters are restricted to integers. Hence, the known exact solution methods are expected to fail on practically relevant problem sizes. In this work we derive a fast heuristic algorithm by approximating the feasible region of the CE problem with a decision tree for which CEs can be calculated more quickly. To this end we analyze the mathematical structure of the feasible region and develop a hand-crafted decision tree which applies splits imitating this structure. Additionally, we develop several techniques to increase the performance of our method. The practicability of our methods is supported by computational experiments on the knapsack problem and the facility location problem.
\end{abstract} 


\section{Introduction}

\subsection{Motivation}

Mathematical optimization is a central tool for decision making in Operations Research and Management Science and is applied in a wide range of industries covering sectors as logistics, health-care, finance or disaster management. Often, exact solution methods for a certain optimization model can be developed and implemented, which provide an optimal solution of the corresponding model. However, the achieved ``optimality'' can give a false sense of explainability. First, the same real-world problem might be approached by different optimization models involving different constraints, objectives or parameter estimations, leading to different optimal solutions. Second, most of the exact algorithms still lack meaningful explanations for stakeholders which are not technical experts \cite{aydin2026explainable}. Hence, explanations which provide insights on the behaviour of the model and corresponding algorithm for different stakeholders are needed \cite{aydin2026explainable,zhang2025decision}. One concept gaining increasing attention which can be used to generate comprehensible explanations for an optimization model are counterfactual explanations (CEs) \cite{korikov2021counterfactual,kurtz2026counterfactual,engelhardt2025counterfactual}. Given the currently applied optimization model and an optimal solution which has property A, a counterfactual question asks which changes of the problem parameters have to be made such that a solution with property B is optimal. As an example, consider a facility location problem which is used to decide in which districts a health-care facility is opened. Given the optimal decision to open a set of facilities, a stakeholder might ask the question: ``Why is there no facility opened in our district?'' A counterfactual explanation can give an answer such as, ``If the population of your district were $10\%$ higher, the model would have opened a facility in your district.''

Only a few methods exist which are able to calculate CEs for problems with integer decision variables. The works of Korikov and Beck study CEs for binary problems where only the objective parameters are allowed to change \cite{korikov2021counterfactual,korikov2023objective}. Later, in \cite{engelhardt2025counterfactual} exact algorithms were developed for the case of binary problems where both objective and constraint parameters are allowed to change. However, the proposed algorithm is computationally tractable only for small instance sizes and problems where the parameters of a single constraint are allowed to change, while the more general setup is not assumed to scale well for realistic problem sizes. The authors prove that the problem even with a single mutable constraint is $\Sigma_2^p$-complete and hence tractable exact algorithms for large-scale problems are currently out of reach.  

\paragraph{Our Contributions:}

In this work, we derive a fast heuristic algorithm to calculate CEs for integer problems where the parameters of multiple constraints are allowed to change. To this end, we analyze the structure of the feasible region of the CE problem and approximate it with a hand-crafted decision tree for which CEs can be calculated quickly. Our main contributions are as follows. (i) We analyze the mathematical structure of the feasible region of the weak CE problem and show that it can be described as a finite union of polyhedra.
(ii) We prove that under mild assumptions, an optimal CE lies on a hyperplane parametrized by a feasible solution of the original problem which we attempt to explain. 
(iii) We develop a decision tree which exploits the latter insight by generating splits which coincide with the hyperplanes potentially containing optimal CEs.
(iv) We derive an algorithm to train such decision trees on randomly generated parameter samples and engineer several  techniques to improve the accuracy of the decision tree and the computation time of our algorithm.
(v) We present experiments on the knapsack problem and the facility location problem demonstrating the capability of the proposed method to generate CEs of good quality significantly faster than available benchmarks.

\subsection{Related Literature}
The concept of counterfactual explanations was first introduced for machine learning models in \cite{wachter2017counterfactual} and has been receiving increasing attention since then; see e.g. the surveys \cite{guidotti2024counterfactual,verma2024counterfactual}.

Later, the concept was translated to optimization problems. In \citep{korikov2021counterfactual,korikov2023objective} counterfactuals for binary linear optimization problems where only objective parameters are mutable were studied, applying techniques from constraint programming and inverse optimization. In \cite{forel2023explainable} the authors study counterfactuals in the context space for a predict-then-optimize approach. Afterward, \cite{kurtz2026counterfactual} identified three types of counterfactuals of interest for optimization problems -- weak, strong, and relative CEs -- and analyzed them for linear optimization problems with continuous decision variables. A fast heuristic based on an alternating direction method was developed in \cite{lefebvre2025computing} to calculate feasible weak CEs for continuous linear optimization problems. Exact methods for weak and strong CEs for integer linear optimization problems were developed in \cite{engelhardt2025counterfactual} for the case where mutable parameters may also appear in the constraints. Relative CEs have since been investigated for contextual stochastic optimization \citep{ramirez2025relative}, and CEs have been used to explain efficiency targets in data envelopment analysis \citep{bogetoft2024counterfactual} and to detect the minimum number of constraints to make a problem infeasible in constraint programming \citep{dasgupta2024counterfactual}.

\paragraph{Interpretability in optimization.} In contrast to the concept of explainability where a justification for the selection of a specific solution is required, interpretability requires the utilized model or algorithm to be inherently interpretable. The design of such models must suit the affected stakeholders' abilities and needs based on their domain knowledge. Due to this requirement, using interpretable methods leads to solutions which can be explained without using further tools. However, the use of interpretable methods often requires simplifications leading to a potential loss in solution quality compared to the use of black-box methods such as industrial MIP solvers. Besides problem-specific approaches \cite{ciocan2022interpretable, bertsimas2021interpretable} and simple heuristics \cite{nguyen2012computational}, the use of interpretable surrogate models has been proposed. The latter include decision trees of restricted size \cite{goerigk2023framework, goerigk2024feature, lumbreras2024explaining} and sparse linear models \cite{goerigk2026linear} mapping instances to solutions. For an extended overview of the current literature on interpretability in optimization we refer to \cite{lumbreras2025interpretable} and \cite{goerigk2024feature}.

\paragraph{Other explainability approaches.} Beyond CEs and the use of interpretable models, other approaches to ensure explainability in mathematical optimization exist. These include the use of feature selection \cite{aigner2025feature}, local explanations \cite{otto2025coherent}, and the enforcement of a solution's similarity to historical decisions \cite{aigner2024framework}. In two-stage stochastic optimization a so-called certificate, a set of representative scenarios, can also be provided in order to  justify solutions \cite{pisinger2026deep}. Although sensitivity analysis \cite{borgonovo2016sensitivity} can be used to provide insights into causal relationships, it is questionable to what extent it actually provides sufficient explanations \cite{aydin2026explainable}. Explanations were also studied for metaheuristics \cite{zhou2024evolutionary,fyvie2025towards,catalano2024explaining}. Additionally, the use of LLMs to produce human-accessible explanations has been recently proposed \cite{chen2025optichat,zhang2025decision,xiao2026deepor}. A summary of the current state-of-the-art methods and a call for interdisciplinary action can be found in \cite{aydin2026explainable}.

\section{General Approach}\label{sec:general}
\subsection{Problem Definition}\label{sub:general:problem}
In this work we study the Weak Counterfactual Explanation Problem (WCEP) applied to general linear integer optimization problems. We define the present problem with present parameters $(\hat c, \hat A, \hat b)\in  \mathbb R^n \times \mathbb R^{m\times n} \times \mathbb R^m$ as
\begin{equation}\label{eq:present_problem}\tag{PP}
    \begin{aligned}
        \min \ & \hat c^\top x \\
        s.t. \quad & \hat A x \ge \hat b \\
        & x\in \mathcal X
    \end{aligned}
\end{equation}
where $\mathcal X\subseteq \mathbb Z_+^n$ and $\mathbb Z_+^n$ is the set of all $n$-dimensional vectors with non-negative and integer entries. \textcolor{black}{We assume that $\mathcal X$ can be expressed through further linear constraints, such as lower and upper bounds.} In line with the definition in \cite{kurtz2026counterfactual} let $\mathcal H \subseteq \mathbb R^n \times \mathbb R^{m\times n} \times \mathbb R^m$ be the compact and convex mutable parameter space with $(\hat c, \hat A, \hat b)\in\mathcal H$ and $\mathcal D\subseteq \mathbb Z_+^n$ the favored solution space. A weak counterfactual question asks for the minimum change of the present parameters, $(\hat c, \hat A, \hat b) \to (c,A,b)$, such that at least one optimal solution of the problem parameterized by the new parameters lies in $\mathcal D$. \textcolor{black}{In the following, we use the notation $[m]=\{1,\ldots,m\}$ to denote index sets.}

\begin{example}
Consider the capacitated facility location problem where we have a set of potential facility locations $\mathcal I$ and a set of customers $\mathcal J$. Opening a facility in location $i\in \mathcal I$ comes with a cost of $f_i$ and each facility has a capacity of $C_i$. Each customer $j \in \mathcal J$ has a demand of $d_j$ which has to be served from the opened facilities. The costs for serving customer $j$ from facility $i$ is $c_{ij}$. This problem can be formulated as
\begin{align*}
\min \ & \sum_{i \in \mathcal I} \sum_{j \in \mathcal J} c_{ij} y_{ij} + \sum_{i \in \mathcal I} f_i x_i\\
&\sum_{i \in \mathcal I} y_{ij} \ge d_j && \forall j \in \mathcal J  \\
&\sum_{j \in \mathcal J} -y_{ij} + C_i x_i \geq 0&& \forall i \in \mathcal I  \\
&y_{ij} \in \mathbb{Z}_+ && \forall i \in \mathcal I, j \in \mathcal J \\
&x_i \in \{0,1\} && \forall i \in \mathcal I
\end{align*}
where $y_{ij}$ is the amount of $j$'s demand served by facility $i$ and $x_i$ has value $1$ if facility $i$ is opened and $0$ otherwise. In an optimal solution $(x^*,y^*)$ of the problem it may be that a certain district which is covering the locations $\mathcal I_D\subset \mathcal I$ does not get any facility, i.e., $x_i^* = 0$ for all $i\in \mathcal I_D$. The following counterfactual question could be asked: ``What is the minimum change in the demand values $d_j$ such that at least one facility will be opened in district $\mathcal I_D$?'' In this case the favored solution space is defined as $\mathcal D=\{ (x,y): \sum_{i\in \mathcal I_D}x_i \ge 1\}$\textcolor{black}{. Each demand value $d_j$ is mutable in $d_j\in [0,\overline d_j]$ with $\overline d_j$ being some given upper bound.}
\end{example}

Formally, a weak CE is defined as follows.
\begin{definition}
A point $(c, A, b)\in \mathcal H$ is a weak counterfactual explanation if
\[
\argmin_{z\in\mathcal X: Az \ge b} c^\top z \ \cap \ \mathcal D \neq \emptyset .
\]
\end{definition}
An optimal weak CE is a weak CE that minimizes the distance to the present parameters implied by norm $\| \cdot \|$. It can be calculated by the following problem,
\begin{equation}\label{eq:weak_CE}\tag{WCEP}
\begin{aligned}
    \inf \ & \|(\hat c, \hat A, \hat b) - (c, A, b)\|\\
    s.t. \ & x \in \argmin_{z\in\mathcal X: Az \ge b} \ c^\top z  \\
    & x\in\mathcal D \\
    & (c,A,b)\in \mathcal H.
\end{aligned}
\end{equation}
We use the infimum operator here, since the feasible region of the problem can be open (as our following results indicate) and thereby we cannot ensure that \eqref{eq:weak_CE}  has an optimal solution. Due to the compactness of $\mathcal H$ and continuity of $\| \cdot \|$ the problem is either bounded or infeasible. For ease of notation in the following we denote projection of the feasible region of this problem in the $(c,A,b)$-space as $\mathcal W$.  

\eqref{eq:weak_CE} can be interpreted as an optimistic bilevel problem and it was shown in \cite{kurtz2026counterfactual} that even for continuous present problems the CE problem can have several undesirable structural properties (such as open, non-convex or even disconnected feasible regions) which make the problem computationally hard to solve. For integer decision variables the problem is $\Sigma_2^p$-hard even for a single mutable constraint \cite{engelhardt2025counterfactual}. The latter results indicate that solving \eqref{eq:weak_CE} to optimality in realistic problem dimensions is currently not within reach and hence fast heuristic methods are needed. In this work we present a decision-tree-based heuristic method which calculates feasible solutions for \eqref{eq:weak_CE} in the integer case with multiple mutable constraints. 

\subsection{Analysis of the Feasible Region}\label{sub:general:feas}
In the following we derive a mathematical description of the feasible region of \eqref{eq:weak_CE} which we later use to develop hand-crafted decision trees as approximate surrogates of the region.

To get an understanding of the structure of the feasible region of \eqref{eq:weak_CE} we consider the following small example. Consider the present problem
\begin{equation}\label{eq:knapsack_example}
\begin{aligned}
\min & \ 3x_1 + 1x_2 + 2x_3 \\
s.t. \quad  & \hat a_1 x_1 + \hat a_2x_2 + 3x_3 \ge 10 \\
& x_1+x_2+x_3 = 2 \\
& x_1,x_2,x_3\in \{ 0,1\},
\end{aligned}
\end{equation}
where only the constraint coefficients $a_1,a_2$ are mutable and the mutable parameter space is the whole $\mathbb R^2$. The favored solution space is defined as $\mathcal D = \{ x\in \{ 0,1\}^3: x_1+x_2 = 1\}$. We can now go through every possible feasible solution $x\in \mathcal X = \{ (1,1,0), (1,0,1), (0,1,1)\}$ and define the region $\mathcal F(x)$ in the $(a_1,a_2)$-space for which the corresponding solution $x$ is feasible. We obtain
\begin{align*}
\mathcal F((1,1,0)) &= \{ (a_1,a_2): a_1+a_2 \ge 10\}, \\
\mathcal F((1,0,1)) &= \{ (a_1,a_2): a_1 \ge 7\}, \\
\mathcal F((0,1,1))& = \{ (a_1,a_2): a_2 \ge 7\},
\end{align*}
i.e., every region is a polyhedron in the $(a_1,a_2)$-space with a boundary defined by hyperplanes with normal vector $(x_1,x_2)$ coming from a feasible solution $x$. The corresponding hyperplanes lead to a partition of the $(a_1,a_2)$-space into cells, where in each cell a fixed set of solutions from $\mathcal X$ is feasible; see Figure \ref{fig:example_feasible_region}. For each of these solutions which are feasible in a certain cell (which could be the empty set), we can calculate the corresponding objective value, which is fixed since the objective coefficients are not mutable. If a solution from the favored solution space $\mathcal D$ is optimal in the cell, the corresponding region is feasible, otherwise it is infeasible. For the given objective parameters in \eqref{eq:knapsack_example} this leads to the feasible region shown in blue in Figure \ref{fig:example_feasible_region}. This feasible region is a union of polyhedra.
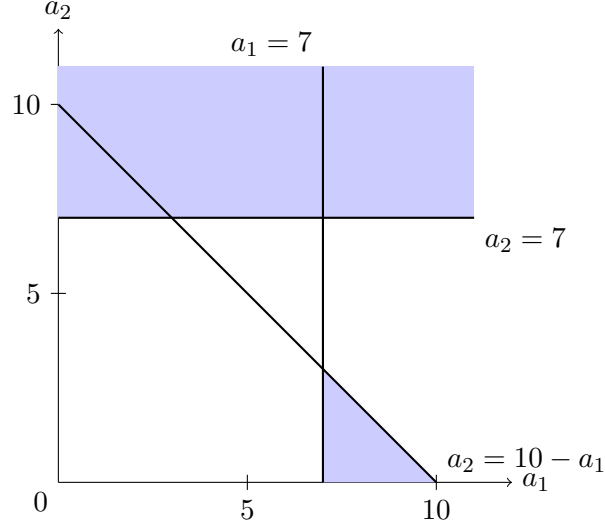
\begin{figure}
    \centering
    \begin{tikzpicture}[scale=0.5]
  \draw[->] (0,0) -- (12,0) node[right] {$a_1$};
  \draw[->] (0,0) -- (0,12) node[above] {$a_2$};
  \foreach \x in {5,10}
    \draw (\x,0.2) -- (\x,-0.2) node[below] {\x};
  \foreach \y in {5,10}
    \draw (0.2,\y) -- (-0.2,\y) node[left] {\y};
  \node[below left] at (0,0) {0};
  \fill[blue!20] (0,7) -- (0,11) -- (11,11) -- (11,7) -- cycle;
  \fill[blue!20] (7,0) -- (10,0) -- (7,3) -- cycle;
   \draw[thick,black] (0,10) -- (10,0) node[above right] {$a_2 = 10 - a_1$};
  \draw[thick,black] (7,0) -- (7,11) node[above left] {$a_1 = 7$};
  \draw[thick,black] (0,7) -- (11,7) node[below right] {$a_2 = 7$};
\end{tikzpicture}
    \caption{In blue the feasible region of the \eqref{eq:weak_CE} for present problem \eqref{eq:knapsack_example}.}
    \label{fig:example_feasible_region}
\end{figure}

The following theorem shows that this structure holds in general.
\begin{lemma}\label{lem:feasible_region}
    The feasible region of \eqref{eq:weak_CE} in the $(c,A,b)$-space is given as
    \begin{equation}\label{eq:feasible_region_formular}
    \mathcal W = \bigcup_{x\in \mathcal D}\ \bigcup_{\mathcal Y\subseteq \mathcal X: x\in \mathcal Y}\ \bigcup_{\sigma: \mathcal X\setminus \mathcal Y \to [m]}  \mathcal F(x,\mathcal Y, \sigma) \cap \mathcal H
    \end{equation}
    where 
\begin{equation*}
\begin{aligned}
\mathcal F(x,\mathcal Y, \sigma) = \big\{ (c,A,b):\  & Ay\ge b \ && \forall y\in\mathcal Y, \\
& c^\top x \le c^\top y \ && \forall y\in \mathcal Y, \\
& a_{\sigma(y)}^\top y < b_{\sigma(y)} \ && \forall y\in \mathcal X\setminus \mathcal Y \ \big\} 
\end{aligned}
\end{equation*}
and $a_{\sigma(y)}$ is the row of $A$ with index $\sigma (y)$.
\end{lemma}
\begin{proof}
We first show that every feasible point $(c,A,b)$ of \eqref{eq:weak_CE} must lie in one of the sets $\mathcal F(x,\mathcal Y, \sigma)$. Let $(c,A,b)\in\mathcal H$ be an arbitrary feasible point of \eqref{eq:weak_CE}. Let $\mathcal Y'$ be the non-empty set of feasible solutions for the given parameters, i.e.
\[
\mathcal Y':=\{ y\in \mathcal X: Ay\ge b\}.
\]
Since $(c,A,b)$ is feasible for \eqref{eq:weak_CE} there must exists one solution from $\mathcal D$ which is in $\mathcal Y'$ and which is optimal under objective coefficients $c$ over all solutions in $\mathcal Y'$. We denote such solution as $x'$. Hence, it holds $Ay\ge b$ and $c^\top x' \le c^\top y$ for all $y\in \mathcal Y'$. Furthermore, all solutions $y\in\mathcal X\setminus \mathcal Y'$ must be infeasible under constraint parameters $A,b$, i.e., for every $y\in\mathcal X\setminus \mathcal Y'$ there exists an index $\sigma' (y)\in [m]$ for which the corresponding constraint is violated, i.e., $a_{\sigma' (y)}^\top y < b_{\sigma' (y)}$. It follows that $(c,A,b)\in \mathcal F(x', \mathcal Y', \sigma')$.

We now show that every point in the set \eqref{eq:feasible_region_formular} is feasible for \eqref{eq:weak_CE}. Let $(c,A,b) \in \mathcal F(x'',\mathcal Y'', \sigma'')\cap \mathcal H$ for some $x''\in\mathcal D$, $\mathcal Y''\subseteq \mathcal X$ and $\sigma'': \mathcal X\setminus \mathcal Y'' \to [m]$. By definition $(c,A,b)\in\mathcal H$. If we can show that $x''$ is an optimal solution of problem 
\begin{equation}\label{eq:linear_prob_proof_theorem}
\begin{aligned}
    \min \ & c^\top x \\
    s.t. \quad & Ax\ge b \\
    & x\in\mathcal X,
\end{aligned}
\end{equation}
then we proved the result. Since it holds $a_{\sigma''(y)}^\top y<b_{\sigma''(y)}$ for all $y\in \mathcal X\setminus \mathcal Y''$, all solutions in $\mathcal X\setminus \mathcal Y''$ must be infeasible. Furthermore, since all solutions in $\mathcal Y''$ are feasible by definition of $\mathcal F(x'',\mathcal Y'', \sigma'')$, the latter problem is equivalent to
\begin{align*}
\min \ & c^\top x \\
s.t. \quad & x\in \mathcal Y''.
\end{align*}
Since $x'' \in \mathcal Y''$ and by definition of $\mathcal F(x'',\mathcal Y'', \sigma'')$ we have $c^\top x''\le c^\top y$ for all $y\in \mathcal Y''$, it follows that $x''$ must be an optimal solution of the latter problem, and hence of \eqref{eq:linear_prob_proof_theorem}.
\end{proof}
The essential insight from this lemma is that the feasible region of \eqref{eq:weak_CE} is a union of sets $\mathcal F(x,\mathcal Y, \sigma)\cap \mathcal H$ and each set $\mathcal F(x,\mathcal Y, \sigma)$ is a (potentially open) polyhedron in the flattened $(c,A,b)$-space. Since an optimal CE is the closest point to $(\hat c, \hat A, \hat b)$ in the feasible region, either its component values must be equal to the present parameters or lie on the boundary of one of the sets $\mathcal F(x,\mathcal Y, \sigma)\cap \mathcal H$. This leads to the following necessary condition for optimal CEs. 
We call the mutable set $\mathcal H$ row-wise, if $\mathcal H = \mathcal H_c \times \mathcal H_1\times \cdots \times \mathcal H_m$ and $(c,A,b) \in \mathcal H$ if and only if $c\in \mathcal H_c$ and $(a_i,b_i)\in \mathcal H_i$ for all $i\in [m]$ where $a_i$ is the $i$-th row of $A$.
\begin{theorem}[Necessary Condition] \label{thm:necessary_condition}
Assume $\| \cdot\|$ is a strictly monotone norm, $\mathcal X$ is finite and $\mathcal H$ is compact, row-wise and convex and $(c^*,A^*,b^*)$ minimizes $\| (c,A,b)-(\hat c,\hat A, \hat b)\|$ over the closure of $\mathcal W$. Then, for every $i\in [m]$ either $(a_i^*,b_i^*) = (\hat a_i, \hat b_i)$ or there exists an $x^i\in \mathcal X$ such that $(a_i^*)^\top x^i  = b_i^*$. Furthermore, either $c^* = \hat c$ or there exist $x^0\in \mathcal D$ and $y^0\in \mathcal X$ with $y^0\neq x^0$, such that $(c^*)^\top (x^0-y^0)  = 0$.
\end{theorem}
\begin{proof}
Let $(c^*,A^*,b^*)$ be a minimizer of $\| (c,A,b)-(\hat c,\hat A, \hat b)\|$ over the closure of $\mathcal W$. This minimizer exists since $\mathcal H$ is compact and hence the closure of the feasible region must be compact as well. Then by Lemma \ref{lem:feasible_region} there exists an $x\in \mathcal D$, $\mathcal Y\subseteq \mathcal X$ with $x\in \mathcal Y$ and $\sigma: \mathcal X\setminus \mathcal Y \to [m]$ such that $(c^*,A^*,b^*)\in \cl \left(\mathcal F(x,\mathcal Y, \sigma)\cap \mathcal H\right)$. First note, that if $\mathcal H \subseteq \mathcal F(x,\mathcal Y, \sigma)$, then $(\hat c, \hat A, \hat b)\in \mathcal F(x,\mathcal Y, \sigma)\cap \mathcal H$ and hence $(c^*,A^*,b^*) = (\hat c, \hat A, \hat b)$ since this leads to a distance of zero. Hence, in this case the necessary condition is true. 

Now consider the opposite case that $\mathcal F(x,\mathcal Y, \sigma)\cap \mathcal H \neq \mathcal H$. Due to the row-wise property of $\mathcal H$ we can rewrite
\[
 \mathcal F(x,\mathcal Y, \sigma)\cap \mathcal H = \left(\mathcal F_c(x,\mathcal Y, \sigma)\cap \mathcal H_c \right) \times\left( \mathcal F_1(x,\mathcal Y, \sigma)\cap \mathcal H_1 \right)\times \cdots \times \left( \mathcal F_m(x,\mathcal Y, \sigma)\cap \mathcal H_m \right)
\]
where
\begin{align*}
\mathcal F_c(x,\mathcal Y, \sigma) & = \left\{ c: c^\top x - c^\top y \le 0 \ \forall y\in \mathcal Y\right\} \\
\mathcal F_i(x,\mathcal Y, \sigma) & = \left\{ (a,b): a^\top y \ge b \ \forall y\in\mathcal Y, \ a^\top y < b \ \forall y\in\mathcal X\setminus \mathcal Y : \sigma(y)=i\right\} \quad \forall i\in [m],
\end{align*}
i.e., we can consider each row of the problem (including the objective function) independently. Since $(c^*,A^*,b^*)$ is a minimizer of the distance to $(\hat c, \hat A, \hat b)$ in $\cl\left(\mathcal F(x,\mathcal Y, \sigma)\cap \mathcal H\right)$ it must hold: if $(\hat a_i,\hat b_i)\in \cl\left( \mathcal F_i(x,\mathcal Y, \sigma)\cap \mathcal H_i\right)$, then $(a_i^*,b_i^*)=(\hat a_i,\hat b_i)$ due to the strict monotonicity of the norm. We will show in the following, if that is not the case, then there must exist an $x^i\in\mathcal X$ such that $(a_i^*)^\top x^i  = b_i^*$. Assume the opposite, that such an $x^i$ does not exist. Then, since $(a_i^*,b_i^*)\in \cl\left( \mathcal F_i(x,\mathcal Y, \sigma)\right)$, it must hold $(a_i^*)^\top y - b_i^* > 0$ for all $y\in\mathcal Y$ and $(a_i^*)^\top y - b_i^*< 0$ for all $y\in\mathcal X\setminus \mathcal Y$ with $\sigma(y)=i$ and it follows $(a_i^*,b_i^*)\in \mathcal F_i(x,\mathcal Y, \sigma)$. Furthermore, since $\mathcal X$ is finite there exists a ball with radius $\varepsilon >0$ in the norm $\| \cdot \|$ around  $(a_i^*,b_i^*)$, denoted as $\mathcal B_\varepsilon\left( (a_i^*,b_i^*)\right)$, such that $\mathcal B_\varepsilon\left( (a_i^*,b_i^*)\right)\subseteq \mathcal F_i(x,\mathcal Y, \sigma)$. Since $\mathcal H$ is convex we know that the whole line
\[
\lambda (a_i^*,b_i^*) + (1-\lambda) (\hat a_i,\hat b_i) \in \mathcal H_i
\]
for all $\lambda\in [0,1]$. Hence, there exists a point on the line with $0<\lambda^*<1$ which lies in $\mathcal B_\varepsilon\left( (a_i^*,b_i^*)\right) \cap  \mathcal F_i(x,\mathcal Y, \sigma)\cap\mathcal H_i$, and
\[
\| \lambda^* (a_i^*,b_i^*) + (1-\lambda^*) (\hat a_i,\hat b_i) - (\hat a_i,\hat b_i) \| = \lambda^* \| (a_i^*,b_i^*) - (\hat a_i,\hat b_i) \| < \| (a_i^*,b_i^*) - (\hat a_i,\hat b_i) \| 
\]
where the latter inequality again follows from the strict monotonicity of the norm. This shows that we found a feasible point which has better objective value than the minimizer which is a contradiction.

The same argumentation can be used to show that the necessary condition on the $c$-parameters is true.
\end{proof}

\subsection{Tree-Approximations of the Feasible Region}\label{sub:general:tree}
In the following we assume that the objective parameters $\hat c$ are immutable and the mutable parameter space is restricted to the $(A,b)$-space, i.e., a CE is a point $(A,b)\in \mathcal H\subseteq \mathbb R^{m\times n} \times \mathbb R^m$. Note that this assumption is not restrictive since we can always transform the present problem via an epigraph reformulation into the problem
\begin{align*}
        \min \ & \mu \\
    s.t. \quad &  \mu - \hat c^\top x \ge 0 \\
    & \hat Ax\ge \hat b \\
    & x\in\mathcal X, \ \mu\in\mathbb R
\end{align*}
and interpret the objective parameters $\hat c$ as constraint parameters. However, note that the artificial variable $\mu$ is continuous and to apply Theorem \ref{thm:necessary_condition} we need a finite solution space and hence $\mu$ has to be restricted to be integer.

The main idea of our method is to approximate the feasible region of Problem \eqref{eq:weak_CE} in the $(A,b)$-space with a CE-tailored decision tree, trained on parameter samples $\{(A^i,b^i)\}_{i=1}^{N}$. The tree decides for every point $(A,b)\in\mathcal H$ if the point is a weak counterfactual or not.  Following the analysis of the previous section, ideally the leaves of the tree should be the sets $\mathcal F(x,\mathcal Y,\sigma)$. Note that each of these sets, for fixed objective parameters $c$, is given by a finite number of constraints which are parametrized by solutions $y\in \mathcal X$. However, since the number of sets is exponential in the input, a tree of this size will not be tractable. Theorem~\ref{thm:necessary_condition} shows that an optimal solution of \eqref{eq:weak_CE} either is equal to the present parameter point or there exist $x^i\in\mathcal X$ such that it lies on the hyperplanes $\{ (A,b)\in \mathcal H: a_i^\top x^i = b_i\}$. These two observations motivate to train decision trees where each split is given by a solution $y\in \mathcal X$ and checks for a given point $(A,b)$ if $Ay\ge b$ or not. In case the answer is yes we continue with the left child-node and otherwise with the right child-node. For a given leaf let $\mathcal Y_F$ be the set of split-solutions $y_F$ for which we went to the left child-node and $\mathcal Y_I$ is the set of split-solutions $y_I$ for which we went to the right child-node. Then the corresponding leaf is the set
\[
\mathcal L(\mathcal Y_F,\mathcal Y_I):=\{ (A,b): Ay_F\ge b \ \forall y_F \in \mathcal Y_F, \  A y_I \ngeq b \ \forall y_I \in \mathcal Y_I \}.
\]
Each leaf is labeled by the decision tree either as \textit{feasible} or \textit{infeasible}. The approximated feasible region of \eqref{eq:weak_CE} given by the tree is the union of all leaves $\mathcal L(\mathcal Y_F,\mathcal Y_I)$ which are labeled as \textit{feasible}. In Figure \ref{fig:visualization_tree} we show such a tree which exactly recovers the feasible region of the present problem given in \eqref{eq:knapsack_example}.
\begin{figure}
    \centering
    \small
    \begin{tikzpicture}[
        >=stealth, line width=0.8pt,
        split/.style={rectangle, rounded corners, draw=black, align=center,
                      inner sep=4pt, minimum height=1cm},
        feas/.style={rectangle, rounded corners, draw=black, fill=blue!15,
                     align=center, inner sep=4pt, minimum height=1cm},
        infeas/.style={rectangle, rounded corners, draw=black, fill=black!8,
                       align=center, inner sep=4pt, minimum height=1cm},
        el/.style={midway, fill=white, inner sep=1.5pt, font=\footnotesize}
    ]
    \node[split]  (n1) at (0,0)     {Split $y^{1}=(0,1,1)$\\[1pt]$A\,y^{1}\ge b$ ? \ {\footnotesize($a_2\ge 7$)}};
    \node[feas]   (f1) at (-2.7,-2) {\textit{feasible}\\[1pt]$\mathcal{L}(\{y^{1}\},\emptyset)$};
    \node[split]  (n2) at (2.7,-2)  {Split $y^{2}=(1,1,0)$\\[1pt]$A\,y^{2}\ge b$ ? \ {\footnotesize($a_1{+}a_2\ge 10$)}};
    \node[infeas] (f2) at (0,-4)    {\textit{infeasible}\\[1pt]$\mathcal{L}(\{y^{2}\},\{y^{1}\})$};
    \node[split]  (n3) at (5.4,-4)  {Split $y^{3}=(1,0,1)$\\[1pt]$A\,y^{3}\ge b$ ? \ {\footnotesize($a_1\ge 7$)}};
    \node[feas]   (f3) at (2.7,-6)  {\textit{feasible}\\[1pt]$\mathcal{L}(\{y^{3}\},\{y^{1},y^{2}\})$};
    \node[infeas] (f4) at (8.1,-6)  {\textit{infeasible}\\[1pt]$\mathcal{L}(\emptyset,\{y^{1},y^{2},y^{3}\})$};
    \draw[->] (n1) -- (f1) node[el,above left]  {yes};
    \draw[->] (n1) -- (n2) node[el,above right] {no};
    \draw[->] (n2) -- (f2) node[el,above left]  {yes};
    \draw[->] (n2) -- (n3) node[el,above right] {no};
    \draw[->] (n3) -- (f3) node[el,above left]  {yes};
    \draw[->] (n3) -- (f4) node[el,above right] {no};
    \end{tikzpicture}
    \caption{Example of a decision tree with solution-based splits for the present problem~\eqref{eq:knapsack_example}. Each internal node tests whether a solution $y\in\mathcal{X}$ is feasible for the queried point $(A,b)$, branching to the left (\emph{yes}) or right (\emph{no}). The feasible leaves (blue) reproduce the feasible region shown in Figure~\ref{fig:example_feasible_region}.}
    \label{fig:visualization_tree}
\end{figure}
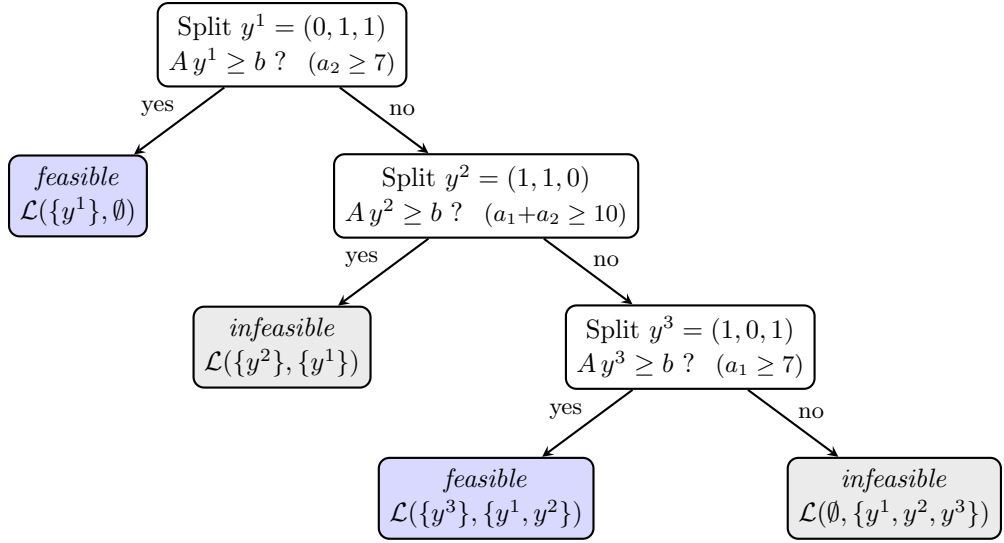

Note that a leaf $\mathcal L(\mathcal Y_F,\mathcal Y_I)$ is usually non-convex due to the conditions $A y_I \ngeq b$. However, it is possible to calculate a CE for the tree as shown in the following.

\paragraph{Calculating a counterfactual explanation for a given decision tree.} An optimal counterfactual explanation over the feasible region given by a fixed decision tree can be calculated by iterating over all leaves of the tree which are labeled as \textit{feasible} and calculate the points in the leaf which minimizes $\| (A,b)-(\hat A,\hat b)\|$. More precisely we solve the problem 
\[
\min_{\substack{\text{feasible} \\ \text{leaf }  \mathcal L}} \inf_{(A,b)\in \mathcal L\cap \mathcal H} \| (A,b)-(\hat A,\hat b)\|.
\]
For a given feasible leaf $\mathcal L:=\mathcal L(\mathcal Y_F,\mathcal Y_I)$ the inner minimization problem can be 
\textcolor{black}{approximated}
as 
\begin{subequations}
\begin{align}
    \min \ &   \| (\hat A, \hat b) - (A, b) \| \\
    s.t. \ & a_i^\top y_F \geq b_i & \forall i \in [m], y_F \in \mathcal Y_F & \label{eq:find_ce_for_tree:feas} \\
           & a_i^\top y_I \leq b_i + Mz_{i,y_I} - \epsilon & \forall i \in [m], y_I \in \mathcal Y_I & \label{eq:find_ce_for_tree:infeas1}\\
           & \sum_{i \in [m]} z_{i,y_I} \leq m-1 & \forall y_I \in \mathcal Y_I \label{eq:find_ce_for_tree:infeas2} \\
           & (A, b) \in \mathcal H \label{eq:find_ce_for_tree:mut_space}\\
           & z_{i,y_I} \in \{ 0,1\} & \forall i \in [m], y_I \in \mathcal Y_I .
\end{align}\label{eq:find_ce_for_tree}
\end{subequations}

Here Constraints \eqref{eq:find_ce_for_tree_row:feas} ensure that any solution $y_F\in \mathcal Y_F$ is feasible for the calculated counterfactual $(A,b)$. In contrast, Constraints~(\ref{eq:find_ce_for_tree:infeas1}) together with (\ref{eq:find_ce_for_tree:infeas2}) ensure that for every infeasible split given by $y_I\in\mathcal Y_I$ at least one of the constraints is violated, i.e., $y_I$ is infeasible for $(A,b)$. The big-M value ensures that in case the binary variable $z_{i,y_I}$ has value $1$, the corresponding constraint in \ref{eq:find_ce_for_tree:infeas1} is redundant. \textcolor{black}{An upper bound for a valid value of $M$ is given by $M:=\max_{i \in [m], y_I \in \mathcal Y_I,(A,b)\in \mathcal H} a_i^\top y_I-b_i+\epsilon$.}
Note that we are using a small value $\epsilon>0$ to ensure strict constraint violation $a_i^\top y_I < b_i$. In practice, \textcolor{black}{numerical difficulties must be taken into account when selecting this value, see the discussion in Section~\ref{sub:algs:pp}}. Finally, \eqref{eq:find_ce_for_tree:mut_space} ensures that the calculated $(A,b)$ point lies in the mutable parameter space $\mathcal H$.

\section{Detailed Algorithm}\label{sec:algs}
Our algorithm follows the steps shown in Figure \ref{fig:flow_chart}, which are described in detail in the following subsections. We first sample a training set of points $\mathcal S= \{ (A^i,b^i)\}^N_{i=1}$ by using a dynamic sampling scheme where each point is labeled either as feasible or infeasible. We can check if $(A^i,b^i)$ is a feasible CE by checking if
\begin{equation}\label{eq:check_if_CE}
\min \{ \hat c^\top x: A^i x \ge b^i, x\in\mathcal X\} =  \min \{ \hat c^\top x: A^i x \ge b^i, x\in\mathcal X\cap \mathcal D\}.
\end{equation}
After sampling and labeling, we apply a preprocessing that enriches the set of samples by resampling around valid CEs. We then train a decision tree $T$ using the CE-tailored splits as described in the previous subsection. Subsequently, we calculate CEs for each leaf of the tree labeled as feasible by solving problem \eqref{eq:find_ce_for_tree}. Afterward we apply post-processing which includes further resampling to account for points which might lie exactly on the feasibility boundary. For this purpose, three different line search approaches are utilized. Finally, we return the best CE found, i.e. the one which is closest to $(\hat A, \hat b)$.

\begin{figure}[h!]
    \centering
    \scalebox{0.7}{\begin{tikzpicture}[
        box/.style={
            rectangle, rounded corners, draw=black,
            text width=2.7cm, minimum height=1.6cm,
            align=center, inner sep=4pt
        },
        node distance=0.7cm,
        >=stealth
    ]
    \node (empty1) {};
    \node (sample)  [box, right=of empty1] {\textbf{Sampling} Sample $\mathcal S= \{ (A^i,b^i)\}^N_{i=1}$ and calculate labels};
    \node (prep)    [box, right=of sample]  {\textbf{Pre-Processing}\\Resample around already sampled CEs};
    \node (approx)  [box, right=of prep]     {\textbf{Training}\\Train decision tree to approximate feasible region $\mathcal{W}$};
    \node (calc_ce) [box, right=of approx]   {\textbf{Calculate CEs}\\Calculate CEs for the decision tree};
    \node (postp)   [box, right=of calc_ce]  {\textbf{Post-Processing}\\Find closer feasible CEs with line search approaches.};
    \node (empty2) [right=of postp] {};
    \draw [thick,->] (sample)  -- (prep);
    \draw [thick,->] (prep)    -- (approx);
    \draw [thick,->] (approx)  -- (calc_ce);
    \draw [thick,->] (calc_ce) -- (postp);
    \draw [thick,->] (postp)   -- (empty2);
    \node (s1) at (4.2,.3) {$\mathcal S$};
    \node (s2) at (7.9,.3) {$\mathcal S$};
    \node (t) at (11.6,.3) {$T$};
    \node (s3) at (15.3,.3) {$\mathcal S^T$};
    \node (abstar) at (20,0) {$(A^*,b^*)$};
    \end{tikzpicture}}
    \caption{Flow chart of our algorithm. The set $\mathcal S$ denotes the set of samples, $T$ is a fitted decision tree, and $\mathcal S^T$ is the set of samples calculated for the decisions tree.}
    \label{fig:flow_chart}
\end{figure}
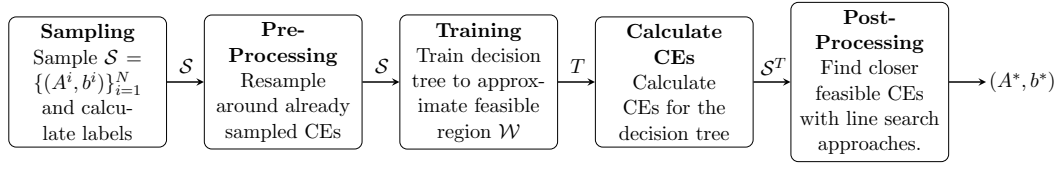

Additionally, Subsection~\ref{sub:algs:cart} contains a heuristic based on an out-of-the-box method to train a decision tree on the $(A,b)$ sample set which does not exploit the observation regarding the feasible region of \eqref{eq:weak_CE} described in Section~\ref{sub:general:feas}. This approach serves as a benchmark in our computational experiments as it is expected to be fast to solve and easy to implement.

Note that the findings described in Section~\ref{sub:general:feas} allow us to build decision trees using two different types of queries. First, we can query whether a solution is feasible for a given point $(A,b)$, we refer to trees using this approach as \textit{matrix-based}. Second, we can query whether a solution satisfies only a single constraint $i \in [m]$. The latter case is referred to as \textit{row-based}. In Section~\ref{sub:algs:samx}, we describe only the approach for the matrix-based case. The procedures for the row-based method, which only requires minor changes to the algorithms and formulations, are described in Appendix~\ref{app:rowbased}.

\subsection{Sampling and Pre-Processing}\label{sub:algs:sampling}
Our algorithm first checks if the present parameters $(\hat A,\hat b)$ are a feasible CE by using \eqref{eq:check_if_CE}. If this is the case we can directly return $(\hat A,\hat b)$ as the optimal CE, since it has distance zero. Hence, in the following we assume that $(\hat A,\hat b)$ is not a feasible CE.

\paragraph{Sampling.}
In the first stage of our method, a dynamic sampling mechanism is applied which draws samples $\mathcal S= \{(A^i,b^i)\}_{i=1}^{N}$ around the parameters $(\hat A,\hat b)$ of the present problem. While sampling in the immediate neighborhood of $(\hat A, \hat b)$ seems tempting, since we are interested in finding the closest CE, we might only consider an area which contains no feasible solution for \eqref{eq:weak_CE}. In contrast, sampling in a larger area ensures proper exploration and increases the likelihood of successfully finding a CE, but also increases the risk that it will be far from the optimal solution of \eqref{eq:weak_CE}.

\begin{algorithm}[h!]
\begin{algorithmic}[1]
\State ce\_found $\gets$ False
\ForAll{$\delta\in \Delta$} \Comment{If no CE was found, restart sampling}
\State $\mathcal S \gets \emptyset$
\RepeatTimes{N} \Comment{Draw $N$ samples}
\State $(A,b) \gets$ sample from $D(\mathcal H,(\hat A, \hat b), \delta)$ \label{eq:sampling:sample}
\State $\mathcal S \gets \mathcal S \cup (A,b)$
\If{$(A,b) \in \mathcal W$} \label{eq:sampling:verify}
\State ce\_found $\gets$ True
\EndIf
\EndRepeat
\If{ce\_found} 
\State \Return $\mathcal S$
\EndIf
\EndFor
\end{algorithmic}
\caption{\label{alg:dyn_sampling}Dynamic sampling.}
\end{algorithm}

We tackle this problem by making use of Algorithm~\ref{alg:dyn_sampling} which iteratively increases the spread of the distributions used. Let $D(\mathcal H,(\hat A, \hat b),\delta)$ be a distribution supported on $\mathcal H$ which is centered around $(\hat A, \hat b)$ and its parametrization changes only by the single value $\delta$. We consider values $\delta \in \Delta := \{ \delta_1,\ldots ,\delta_K\}$ where we define $K$ as an upper limit on the number of iterations of the algorithm. As part of each iteration, we generate $N$ samples (Line~\ref{eq:sampling:sample}) from the current distribution $D(\mathcal H,(\hat A, \hat b),\delta)$. All samples are checked for being a feasible solution for \eqref{eq:weak_CE} (Line~\ref{eq:sampling:verify}) which can be done by checking \eqref{eq:check_if_CE}. If at least one sample is found to be a CE, the algorithm terminates after the respective iteration. Alternatively, if no CE was found, we consider the next value for $\delta$, and we sample again $N$ points in the next iteration using the new distribution. Only the samples found in the last iteration of the algorithm are returned, as an increasing number of samples increases the computational burden on subsequent steps of our method. In our experiments we use (i) uniform distributions where the support set is iteratively increased in each iteration, and (ii), truncated normal distribution where the variance is iteratively increased.

There are many possible adjustments or modifications to this sampling approach. These include, among others, reversing the order of distributions used, i.e. starting with a large spread which decreases over iterations. Depending on the specific underlying problem, $\mathcal H$, and domain knowledge of the decision maker, there might be known feasible solutions to \eqref{eq:weak_CE}. This allows the sample set $\mathcal S$ to be enriched by these data points.

\paragraph{Pre-Processing.} Our proposed pre-processing approach, denoted by \REpre, involves resampling around the CEs found during dynamic sampling described above. For each valid CE found, $N_{pre}$ points are drawn in its close neighborhood and each is labeled as feasible CE or infeasible CE as well by checking \eqref{eq:check_if_CE}. The approach is intended to enable a more detailed exploration of the region around CEs. Furthermore, in case a sampled CE lies close to a boundary of the feasible region, sampling in the close neighborhood can guide the decision tree to find the right splits. The samples generated in this pre-processing are added to the set of samples $\mathcal S$.

\subsection{Training}\label{sub:algs:samx}
In this section we describe a greedy heuristic to determine a decision tree classifying points $(A,b)$ based on whether they are feasible CEs or not, i.e., whether they are contained in $\mathcal W$ or not. Utilizing Theorem~\ref{thm:necessary_condition}, solution vectors $y\in \mathcal X$ of the underlying optimization problem are used as splits. Following the idea of the CART algorithm \cite{breiman1984classification}, the tree is built by choosing iteratively at each query node the best split w.r.t. to the gini impurity criterion. In contrast to the application of the CART algorithm, where a reasonable pool of splits can easily be precomputed, determining them in our application is challenging in itself since each solution in $\mathcal X$ provides a potential split and hence the number of splits might be exponential in the dimension of the present problem. Theoretically, the solution in $\mathcal X$ which minimizes the gini impurity can be calculated by optimizing over $\mathcal X$ with mixed-integer programming methods. In Appendix~\ref{sub:algs:miqp}, we present a non-linear mixed-integer programming formulation to compute splits which are locally optimal. As the formulation involves big-M constraints and a large number of binary variables, this procedure is expected to be too computationally expensive. 

Therefore, instead of finding the optimal split in each iteration of the training process, we select the best split from a pre-computed pool of solutions $X^{cand}\subset \mathcal X$. This procedure utilizes the fact that during the sampling process described in Section~\ref{sub:algs:sampling}, labeling each sample requires to solve $\min \{ \hat c^\top x: A^i x \ge b^i, x\in\mathcal X\}$ and $\min \{ \hat c^\top x: A^i x \ge b^i, x\in\mathcal X\cap \mathcal D\}$ and hence up to two different solution vectors are computed per oracle call which serve as candidates for the splits.

 We introduce the following notation: We define $d$ to be the maximum depth of the tree trained. Therefore, the final tree consists of at most $2^{d}-1$ nodes of which $Q = 2^{d-1}-1$ are potential queries, i.e. nodes representing a test.
 For the sake of simplicity, we make use of the following functions: Let \textit{leftchild(q)} and \textit{rightchild(q)} return the index of the left respective right child node for a given potential query node $q$. Furthermore, for a set of $(A,b)$ samples $\mathcal S$ we define

\begin{equation*}
  gini(\mathcal S) =
  \begin{cases}
    2 \frac{|\{(A,b) \in \mathcal S: (A,b) \in \mathcal W\}|}{|\mathcal S|}  \frac{|\{(A,b) \in \mathcal S: (A,b) \notin \mathcal W\}|}{|\mathcal S|} & \text{if $\mathcal S \neq \emptyset$} \\
    0 & \text{otherwise}
  \end{cases},
\end{equation*}
allowing us to compute the gini impurity of a split resulting in sets $\mathcal S^{\text{left}}$ and $\mathcal S^{\text{right}}$ by 

$$ginisplit(\mathcal S^{\text{left}},\mathcal S^{\text{right}}) = \frac{|\mathcal S^{\text{left}}|}{|\mathcal S^{\text{left}}|+|\mathcal S^{\text{right}}|}gini(\mathcal S^{\text{left}})+ \frac{|\mathcal S^{\text{right}}|}{|\mathcal S^{\text{left}}|+|\mathcal S^{\text{right}}|}gini(\mathcal S^{\text{right}}).$$

\begin{algorithm}[h!]
\begin{algorithmic}[1]
\Ensure samples $\mathcal S$, tree depth $d$, classification threshold $\beta\in (0,1)$ 
\State $\textcolor{black}{X^{cand} \gets \emptyset,} Q \gets 2^{d-1}-1, \mathcal S^1 \gets\mathcal S, \mathcal L \gets \emptyset$
\ForAll{$(A,b) \in \mathcal S$} \Comment{Generate x candidates}
\State $x^1 \gets \argmin \{ \hat c^\top x: A x \ge b, x\in\mathcal X\}$ \label{eq:samx_matrix:oc1}
\State $x^2 \gets \argmin \{ \hat c^\top x: A x \ge b, x\in\mathcal X\cap \mathcal D\}$ \label{eq:samx_matrix:oc2}
\State $X^{cand} \gets X^{cand} \cup \{x^1,x^2\}$
\EndFor
\ForAll{$q \in [Q]$} \Comment{Build tree iteratively} \label{eq:samx_matrix:iter_q}
    \If{$\mathcal S^{q} \neq \emptyset$}
    \State $\mathcal S^{\text{left}} \gets \emptyset, \mathcal S^{\text{right}} \gets \emptyset, g^*\gets0.5, \bar x^q \gets \text{None},\mathcal S^{\text{leftchild(q)}}\gets \emptyset,\mathcal S^{\text{rightchild(q)}}\gets \emptyset$
        \ForAll{$x \in X^{cand}$} \Comment{Determine best query}\label{eq:samx_matrix:iter_x}
            \ForAll{$(A,b) \in \mathcal S^{q}$}
            \If{$Ax \geq b$}
                \State $ \mathcal S^{\text{left}} \gets \mathcal S^{\text{left}}\cup \{(A,b)\}$
            \Else 
                \State $ \mathcal S^{\text{right}} \gets \mathcal S^{\text{right}} \cup \{(A,b)\}$
            \EndIf
            \EndFor
        \If{$\text{ginisplit}(\mathcal S^{\text{left}},\mathcal S^{\text{right}}) < \min\{\text{gini}\text(\mathcal S^{q}),g^*\}$} \label{eq:samx_matrix:newbestsplit}
            \State $\bar x^q \gets x,\mathcal S^{\text{leftchild(q)}}\gets \mathcal S^{\text{left}}, \mathcal S^{\text{rightchild(q)}}\gets \mathcal S^{\text{right}}$ 
            \State $g^* \gets \text{ginisplit}(\mathcal S^{\text{left}}, \mathcal S^{\text{right}})$
        \EndIf
        \EndFor
    \If{$S^{\text{leftchild}(q)} = \emptyset$ \textbf} \Comment{Identify leaf nodes}
        \State $\mathcal L \gets \mathcal L \cup\{q\}$
    \ElsIf{$\text{leftchild}(q) \in \{2^{d-1}+1, ...,2^d-1\}$}
    \State $\mathcal L \gets \mathcal L \cup\{\text{leftchild}(q), \text{rightchild}(q)\}$
    \EndIf
    \EndIf
\EndFor
\State label leaves $\ell \in \mathcal L$ as ``feasible'' if $\frac{|\{(A,b)\in \mathcal S^\ell:(A,b)\in \mathcal W\}|}{|\mathcal S^{\ell}|} \geq \beta$; else ``infeasible'' \label{eq:samx_matrix:labeltree}
\State \Return classification tree on $\mathcal H$ with splits $\bar x^q \ \forall q \in [Q]$ 
\end{algorithmic}
\caption{Greedy algorithm for training the decision tree (matrix based).} 
\label{alg:sample_x}
\end{algorithm}

The training method is presented in Algorithm \ref{alg:sample_x}. First, all sampled points are evaluated for being a CE~(Lines~\ref{eq:samx_matrix:oc1}-\ref{eq:samx_matrix:oc2}) and the computed solutions are stored as split candidates. If used in combination with our dynamic sampling procedure, this part can be incorporated in Algorithm~\ref{alg:dyn_sampling}, Line~\ref{eq:sampling:verify} to avoid repeated computation. For each potential split $q \in [Q]$ (Line~\ref{eq:samx_matrix:iter_q}) we iterate over the solution candidates (Line~\ref{eq:samx_matrix:iter_x}) dividing the set $\mathcal S^q$ into feasible ($\mathcal S^{left}$) and infeasible ($\mathcal S^{right}$) subsets of $(A,b)$ samples. Based on these two sets, the gini impurity for each split is calculated. The solution vector which results in the highest decrease in gini impurity is chosen as split value (Line~\ref{eq:samx_matrix:newbestsplit}). If none of the candidates leads to an improvement, the respective node is regarded as leaf node.
A leaf $i$ is labeled ``feasible'' if more than a fraction of $\beta$ of the samples assigned to it are CEs. Otherwise, a leaf is labeled as ``infeasible'' (Line~\ref{eq:samx_matrix:labeltree}). The algorithm returns the computed decision tree. As described in Section~\ref{sub:general:tree}, CEs can then be computed for all leaves of the tree labeled as ``feasible'' by solving Problem~\eqref{eq:find_ce_for_tree}. In the following, we refer to the set of CEs computed for a decision tree as $\mathcal S^T$. 

\subsection{Post-processing}\label{sub:algs:pp}
Our proposed post-processing approach involves two methods. The first method is a resampling procedure which tries to reduce the numerical instabilities resulting from Problem \eqref{eq:find_ce_for_tree} used to solve the CEs. The second applies several simple line search ideas to find CEs which are even closer to the present parameters $(\hat A,\hat b)$ than the calculated CEs.

\paragraph{Resampling procedure.} Preliminary computational experiments showed that due to the openess of the feasible region of \eqref{eq:weak_CE} numerical instabilities might occur. Especially, when solving \eqref{eq:find_ce_for_tree} to obtain CEs for our trained decision tree, the involved parameter $\varepsilon$ which models strict violation of one of the constraints is hard to tune. This results in CEs returned by Problem~\eqref{eq:find_ce_for_tree} which cannot be identified as CEs afterwards when performing the check \eqref{eq:check_if_CE}. The reason is that these point lie close to the boundary of the feasible region which might be open. Our first post-processing method, which we denote as (\REpost), is applied after the decision tree has been trained. $N_{post}$ new samples are drawn from the neighborhood of each CE that was computed for the decision tree (which are not necessarily feasible solutions for (\ref{eq:weak_CE})). The generated samples are evaluated applying \eqref{eq:check_if_CE} and added to the set of samples $\mathcal S^T$. 

\paragraph{Line search.} The second post-processing method contains three improvement steps based on line search approaches, which are supposed to be applied after the decision tree has been trained and have the goal to find CEs which are even closer to the present parameters $(\hat A,\hat b)$. All three methods take as input a set of samples $\mathcal S^{T}$, which were found by computing CEs for the decision tree, including points found by \REpost, if applicable. While Theorem~\ref{thm:necessary_condition} ensures that there must be a hyperplane, defined by a solution vector, that defines the boundary of the feasible region of (\ref{eq:weak_CE}), the heuristic nature of our approach does not guarantee that a solution we find actually lies on this boundary. The methods presented here therefore aim to identify such a point or to get closer to it.

\begin{figure}[h!tb]
\centering
\begin{subfigure}{.32\textwidth}
        \centering
        \includegraphics[width=\linewidth]{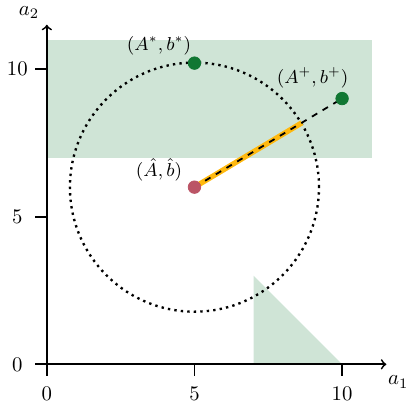}
    \caption{\LSone.}
    \label{fig:ls_pos}
\end{subfigure}
\begin{subfigure}{.32\textwidth}
        \centering
        \includegraphics[width=\linewidth]{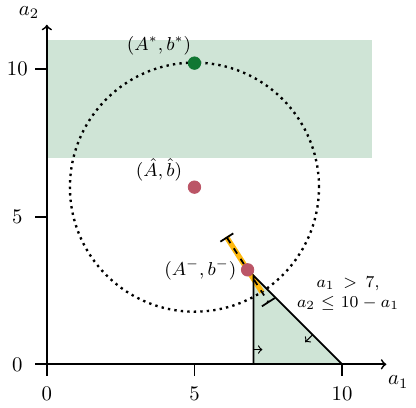}
    \caption{\LStwo.}
    \label{fig:ls_neg1}
\end{subfigure}
\begin{subfigure}{.32\textwidth}
        \centering
        \includegraphics[width=\linewidth]{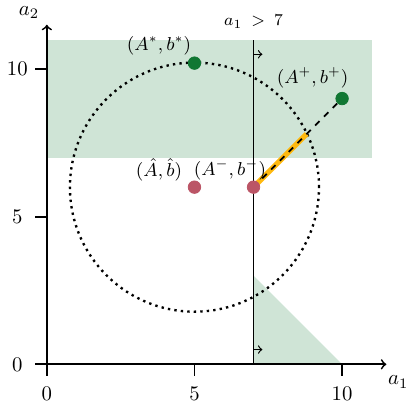}
    \caption{\LSthree.}
    \label{fig:ls_neg2}
\end{subfigure}
\caption{Visualization of the three types of line search applications in $\mathcal H$. The point $(A^*,b^*)$ is the best-known solution and the dotted circle indicates the $L_2$-ball with radius $d^*=\| (A^*,b^*) - (\hat A, \hat b)\|_2$ centered around $(\hat A, \hat b)$. $\mathcal W$ is colored in green. The dashed line represents the line considered, while the part highlighted in orange represents the part of the line which potential candidates for better CEs. }
\label{fig:sketch_ls}
\end{figure}

Figure~\ref{fig:sketch_ls} sketches the intuition of these methods for an example problem with parameters $a_1, a_2$ being mutable in $\mathbb R^2$. Figure~\ref{fig:ls_pos} presents the first approach, referred to as \LSone. We conduct a line search between a feasible CE $(A^+,b^+)$ and the original instance $(\hat A, \hat b)$, indicated by the dashed line, as we know that at least one boundary of the feasible region has to lie between these two points (since otherwise $(\hat A, \hat b)$ is feasible and hence an optimal CE which can be detected before starting our procedure). By shifting $(A^+,b^+)$ towards $(\hat A, \hat b)$ while ensuring to stay within $\mathcal W$, we aim to minimize the distance to this boundary. If $(A^+,b^+)$ lies in the interior of the feasible region this procedure will always result in a closer CE.

The second method (\LStwo) depicted in Figure~\ref{fig:ls_neg1}, addresses numerical issues and rounding errors which may lead to a computed point $(A^-,b^-)$, which is found to be infeasible for (\ref{eq:weak_CE}), being slightly next to a boundary of the feasible region. For each sample $(A^-,b^-)$ not being a CE, we search on the line given by $\lambda (\hat A,\hat b)+(1-\lambda)(A^-, b^-)$ for $\lambda \in [-\alpha, \alpha]$ and some value $\alpha >0$. This corresponds to the intuitive idea of moving the point $(A^-,b^-)$ slightly farther away from or slightly closer to the point $(\hat A, \hat b)$.

For the last method (\LSthree), presented in Figure~\ref{fig:ls_neg2}, we exploit the fact that for a given leaf of the decision tree which was labeled as feasible and for every infeasible point $(A^-,b^-)$ assigned to this leaf at least one border of the feasible region of (\ref{eq:weak_CE}) has to lie between the infeasible point $(A^-,b^-)$ and each feasible point also contained in the leaf. For an infeasible point $(A^-,b^-)$ we therefore determine the leaf to which it is assigned by using our decision tree. We then conduct line search between the point $(A^-,b^-)$ and the best valid CE $(A^+,b^+)$ assigned to the respective leaf.

Algorithm~\ref{alg:line_search} describes formally how to incorporate the three line search methods. We deliberately use the general term line search here as a wide variety of methods can be applied. Possible approaches include simple algorithms such as linear or binary search, or more sophisticated methods such as Bayesian Optimization. We assume that all points considered during the search procedure are evaluated w.r.t. to their property of being a CE. All of these points which satisfy this property are returned in a set with index $+$, all points which are not CEs are returned in a set with index $-$. The latter are used partly for \LSthreesp and are otherwise listed solely for the sake of consistency.
For an illustration of the application using the example of binary search, please refer to Appendix~\ref{app:params}. 

The formulation of Algorithm~\ref{alg:line_search} allows only to conduct a subset of methods presented previously. Note that if \LSthreesp is to be used, but not \LStwo, $\mathcal S_-^{\prime \prime}$ should be replaced by $\mathcal S^T$ in Line~\ref{ls:set}. Furthermore, a limit on the number of iterations in Line~\ref{ls:for1} and/or Line~\ref{ls:for2} might be applied to reduce the run time. 

As we are interested in improving the solution quality measured by the distance between the best feasible solution and $(\hat A, \hat b)$, for each of the three methods we only need to consider the part of the line which is actually closer to $(\hat A, \hat b)$ than the best feasible solution found so far. We indicate the best CE found so far by $(A^*,b^*)$ and define $d^*$ to be $||(A^*,b^*)-(\hat A,\hat b)||$. In all plots, we highlight the $L_2$-ball of radius $d^*$ centered at $(\hat A, \hat b)$ in gray. Note that for methods \LStwosp and \LSthree, it is possible that no part of the line lies within this ball. In this case, no line search is conducted. 
To determine the lower boundary of the respective line segment which is closer to $(\hat A, \hat b)$ than $d^*$ for \LSonesp the formula \begin{equation}\lambda_{lb} = 1- \frac{d^*}{||(A^+,b^+)-(\hat A, \hat b)||}\label{eq::lambda_ls1}\end{equation} is used. 
For \LStwosp and some value $\alpha >0$ we first verify that a line segment within a distance of at most $d^*$ to $(\hat A, \hat b)$ exists, which is true if $1- \frac{d^*}{||(A^-,b^-)-(\hat A, \hat b)||} < \alpha$ holds. In this case, we set $\lambda_{ub} = \alpha$ and compute \begin{equation}\lambda_{lb}=\max\{-\alpha,1-\frac{d^*}{||(A^{-},b^{-})-(\hat A, \hat b)||}\}.\label{eq::lambda_ls2}\end{equation} Otherwise, no line search is executed.
The precalculation of the boundaries for \LSthreesp (Algorithm~\ref{alg:line_search}, Line~\ref{ls:lam_lsn2}) is done by solving $\lambda_{lb}=\argmin \{ \lambda:\lambda \in \mathcal Z \} $ and $\lambda_{ub}=\argmax \{ \lambda:\lambda \in \mathcal Z \} $ for
\begin{align*}
        \mathcal Z = \{ \lambda \in [0,1]: \ & \lambda (A^-,b^-) + (1-\lambda) (A^+,b^+) = (A^\prime,b^\prime), &\\
        & ||(A^\prime,b^\prime)-(\hat A,\hat b)|| \leq d^*,\\ 
        & (A^\prime,b^\prime) \in \mathcal{H}\}.
\end{align*}
If $\mathcal Z = \emptyset$, \LSthreesp is carried out. 

\begin{algorithm}[h]
\begin{algorithmic}[1]
\Ensure samples $\mathcal S, \mathcal S^T$, 
\State $(A^*,b^*) \gets \argmin_{(A,b)\in (\mathcal S \cup S^T)\cap \mathcal W} ||(A,b)-(\hat A, \hat b)||, d^*\gets ||(A^*,b^*)-(\hat A, \hat b)||$, $\mathcal S^\prime_{+/-} \gets \emptyset, \mathcal S^{\prime\prime}_{+/-}\gets \emptyset, \mathcal S^{\prime\prime\prime}_{+/-} \gets \emptyset$ \ForAll{$(A,b) \in \mathcal S^T$} \label{ls:for1}
\If{$(A,b) \in \mathcal W$}
\State compute $\lambda_{lb}$ using Formula (\ref{eq::lambda_ls1})
\State $\mathcal S^\prime_+,\mathcal S^\prime_- \gets$ \textbf{do} line search on $\lambda (\hat A, \hat b)+(1-\lambda)(A,b):\lambda \in [\lambda_{lb},1]$\Comment{Apply \LSone}\label{ls:lsp}
\State update $(A^*,b^*)$, $d^*$
\Else
\State compute $\lambda_{lb}, \lambda_{ub}$ w.r.t. $\alpha$ using Formula (\ref{eq::lambda_ls2})
\State $\mathcal S^{\prime\prime}_+,\mathcal S^{\prime\prime}_-  \gets$ {do} line search on $\lambda (\hat A, \hat b)+(1-\lambda)(A,b):\lambda \in [\lambda_{lb},\lambda_{ub}]$ \Comment{Apply \LStwo}\label{ls:lsn1}
\State update $(A^*,b^*)$, $d^*$
\EndIf
\EndFor
\ForAll{$(A,b) \in \mathcal S^{\prime\prime}_-$} \label{ls:for2}\label{ls:set}
\State $(A^+,b^+)\gets$ best CE assigned to the same leaf as $(A,b)$
\State compute $\lambda_{lb}, \lambda_{ub}$ by solving $\min_{\lambda \in \mathcal Z}$ and $\max_{\lambda \in \mathcal Z}$ \label{ls:lam_lsn2}
\State $\mathcal S^{\prime\prime\prime}_+,\mathcal S^{\prime\prime\prime}_- \gets$ {do} line search on $\lambda (A,b)+(1-\lambda)(A^+, b^+):\lambda \in [\lambda_{lb},\lambda_{ub}]$ \Comment{Apply \LSthree}\label{ls:lsn2}
\State update $(A^*,b^*)$, $d^*$
\EndFor
\State $\mathcal S \gets \mathcal S \cup \mathcal S^T \cup \mathcal S^\prime_+ \cup \mathcal S^{\prime\prime}_+\cup \mathcal S^{\prime\prime\prime}_+$
\State \Return $\mathcal S$
\end{algorithmic}
\caption{\label{alg:line_search}Post-processing: Line search approaches.}
\end{algorithm}

\subsection{Classical CART Algorithm (Benchmark)}\label{sub:algs:cart}
This section describes an algorithm that will serve as a benchmark for the other methods. It is mainly based on the CART algorithm \cite{breiman1984classification}, which is widely available as an out-of-the-box implementation (e.g. in \cite{scikit-learn}). We refer to it as \cart when used in combination with our pre- and post-processing methods, or \cartpuresp if used standalone. 

The method is illustrated in Algorithm~\ref{alg:cart}. As with all methods presented in this paper, we utilize a set of $(A,b)$ samples $\mathcal S$ generated by our dynamic sampling approach. First, we iterate over all samples, creating a set of labels $\mathcal C$, indicating if a point $(A,b)$ is a CE or not (Line~\ref{cart:label}). Next, a classification tree using the CART algorithm and the sets $\mathcal S$ and $\mathcal C$ as training data is constructed (Line~\ref{cart:train}). Potential feasible solutions are obtained by calculating CEs for every leaf of the tree which contains more than a fraction of $\beta$ points which are labeled as CE (Line~\ref{cart:obtain}).

\begin{algorithm}
\begin{algorithmic}[1]
\State $\mathcal C \gets \emptyset, \mathcal T \gets \text{None}$ 
\ForAll{$(A^i,b^i)\in \mathcal S$}
\State $\mathcal{C}(i) \gets \text{``feasible'' if } (A,b) \in \mathcal W, \text{``infeasible'' else}$ 
\label{cart:label}
\EndFor
\State train classification decision tree $\mathcal T$ on $(\mathcal{S},\mathcal{C})$\label{cart:train} with maximum depth $d$
\State $\mathcal S^T \gets$ calculate CEs for all leaves of $\mathcal T$ which contain at least a fraction of $\beta$ samples labeled ``feasible''\label{cart:obtain}
\State \textbf{return} $\argmin \{ ||(A,b)-(\hat A, \hat  b)||:(A,b)\in \mathcal S^T\cap\mathcal W\}$
\end{algorithmic}
\caption{\label{alg:cart} \cart.}
\end{algorithm}

\section{Computational Experiments}

\subsection{Setup}
We conduct experiments for two types optimization problems: Knapsack and the Capacitated Facility Location Problem (CFLP). Whilst the former is mainly used to benchmark against the existing exact method developed in \cite{engelhardt2025counterfactual}, the latter is utilized to analyze the performance on mixed-integer problems\textcolor{black}{. Since problems containing continuous variables are not covered by the theory described in Section~\ref{sec:general}, we will empirically verify to what extent our methods work beyond that limit.} All experiments are based on underlying established instances from the literature. We divide our experiments in four parts, serving the following objectives:
\begin{itemize}
    \item[\ref{sub:knapsack}] Experiments using Knapsack-based instances:
    \begin{itemize}
    \item[\ref{sub:knapsack:sampling}] Determining suitable hyper-parameters for the dynamic sampling.
    \item[\ref{sub:knapsack:algs}] Comparing our approaches to the exact method in \cite{engelhardt2025counterfactual}.
    \item[\ref{sub:knapsack:ablation}] Ablation study to analyze the impact of our pre- and post-processing steps and the tree depth.
    \end{itemize}
    \item[\ref{sub:cflp}] Experiments using CFLP-based instances: Evaluating our method on larger instances of a mixed-integer program, comparing the row- and the matrix-based version.
\end{itemize}

\paragraph{Knapsack-based instances.}
For the knapsack experiments we make use of a set of instances formerly used by \cite{engelhardt2025counterfactual}. These are modified instances of \texttt{kplib}\footnote{\url{https://github.com/likr/kplib}.}, which in turn were created using a method proposed in \cite{kellerer2004exact}. All instances are treated as cover problems in the form of the following MIP formulation:

\begin{align*}
    \min \ & \sum_{i \in [n]} c_i x_i \\
    \text{s.t.} \ & \sum_{i \in [n]} w_i x_i \geq B\\
    & x_i \in \{ 0,1 \} & \forall i \in [n]
\end{align*}

\textcolor{black}{These instances make use of three different types of favored solution spaces: $\mathcal D^+= \{x \in \{0,1 \}^n:x_i=1 \ \ \forall i \in \mathcal I^+ \}$, $\mathcal D^-= \{x \in \{0,1 \}^n:x_i=0 \ \ \forall i \in \mathcal I^- \}$ and $\mathcal D^\geq = \{x \in \{0,1 \}^n:\sum_{i \in \mathcal I^\geq}x_i \geq 1\}$ with $\mathcal I^-$, $\mathcal I^+$, and $\mathcal I^\geq$ being subsets of $[n]$. The weights and the budget are considered mutable within an interval of $\pm5\%$ of the respective original values.} 
In contrast to \cite{engelhardt2025counterfactual}, \textcolor{black}{which only allow the selection of integer values within these intervals}, we also consider fractional values.
\textcolor{black}{Note that this difference impacts the comparability of all experimental results; however, this is the closest possible comparison from all available methods of the state-of-the-art literature.} Suitable instances were identified by filtering out all instances for which the origin instance already has an optimal solution in the favored solution space and all instances for which \cite{engelhardt2025counterfactual} could prove the non-existence of a solution within the mutable parameter space. In total 159 instances were used in the presented experiments, including ten to 40 items each. Not all of them could be solved to certified optimality and for 21 of these, the existence of a feasible solution is still unknown. We always compare to the best-known CE found by \textcolor{black}{\cite{engelhardt2025counterfactual}}.
When referring to \textcolor{black}{those} results we make also use of the term \exact.

\paragraph{CFLP-based instances.}
The CFLP experiments are based on 25 instances taken from the \texttt{ORLIB} library\footnote{\url{https://people.brunel.ac.uk/~mastjjb/jeb/orlib/capinfo.html}} \cite{beasley1988algorithm}. All instances consist of 50 customers, nine include 16, eight include 25 and further eight include 50 facilities. We consider the CFLP problems in the form of the following MIP formulation:
    \begin{align*}
        \min \ & \sum_{i \in \mathcal I} \sum_{j \in \mathcal J} c^{var}_{ij} y_{ij} + \sum_{i \in \mathcal I} f_i x_{i} \\
        \text{s.t.} \ & \sum_{i \in \mathcal I} y_{ij} \geq d_{j} & \forall j \in \mathcal J\\
        & \sum_{j \in \mathcal J} - y_{ij} + C_{i} x_i\geq 0 & \forall i \in I \\
        & y_{ij} \geq 0 &\forall i \in \mathcal I, j \in \mathcal J \\
        & x_i \in \{0,1 \} & \forall i \in \mathcal I
    \end{align*}
Two types of favored solution spaces are considered: $\mathcal{D}^+ = \{ x \in \{0,1\}^{I}: x_i = 1 \ \ \forall i \in \mathcal{I} ^+\}$ and $\mathcal{D}^- = \{ x \in \{0,1\}^{I}: x_i = 0 \ \ \forall i \in \mathcal{I} ^-\}$.
For each of the 25 underlying instances two instances using a favored solution space of type $\mathcal{D^+}$ and two instances using a favored solution space $\mathcal{D^-}$ were derived. Each favored solution space is based on a set $\mathcal{I}^+$ respective $\mathcal{I}^-$ containing a single, randomly selected facility $i \in \mathcal I$, i.e., in each instance the opening or closing of a random facility is demanded. During sampling, only instances were considered for which none of the optimal solutions lie within the favored solution space.
For the mutable parameter space we allow the demands and capacities to be changed by $\pm 100\%$ each. 
To the best of our knowledge, there is currently no other approach which is able to solve \eqref{eq:weak_CE} for these instances which could serve as a benchmark.

\paragraph{Computational Setup.}
The experiments were implemented in Python 3.12.3. Mathematical programming formulations were solved using Gurobi 13.0.1 \cite{gurobi}. For the generation of decision trees as part of the  benchmark method \cartsp we made use of the CART implementation in the library \textit{scikit-learn} \cite{scikit-learn}, version 1.8.0. All computational experiments were executed on a machine with 96 cores running at 3.6 GHz and 256 GB memory. Up to 100 instances were solved in parallel with Gurobi's threads parameter set to one. When solving versions of the underlying problem where the solution is enforced to be in $\mathcal H$, we set Gurobi's internal time limit to one second. For solving Problems (\ref{eq:find_ce_for_tree}), (\ref{eq:find_ce_for_tree_row}), (\ref{eq:gini1}), and (\ref{eq:gini_row}), a time limit of 300 seconds is used.
The code and the considered instances are available on GitHub\footnote{\url{https://github.com/sbstnmrtn/heuristics_counterfactuals_optimization}}. For all experiments, binary search was used as the line search method. Further values of parameters which are not subject to changes during the experiments are reported in Appendix~\ref{app:params}. In addition to the introduction of abbreviations for the methods used when they are first described, a corresponding overview can also be found in Table~\ref{tab:abbreviations}.

\begin{table}[h!]
    \centering
    \begin{tabular}{ll}
        Abbreviation& Description and reference\\
        \hline
        \exact & Exact method; \cite{engelhardt2025counterfactual} \\
        \samx & Training using $X^{cand}$; Section \ref{sub:algs:samx} \\
        \miqp & Training solving MIQPs; Appendix \ref{sub:algs:miqp} \\
        \cart & Training using CART with pre- and post-processing; Section \ref{sub:algs:cart}\\
        \cartpure & Training using CART without pre- and post-processing; Section \ref{sub:algs:cart}\\
        \texttt{only sampling} & Determining CEs by sampling; Section \ref{sub:algs:sampling} \\
        \LSone, \LStwo, \LSthree & Line search approaches; Section \ref{sub:algs:pp} \\
        \REpre, \REpost & Resampling steps; Sections \ref{sub:algs:sampling} \& \ref{sub:algs:pp}
    \end{tabular}
    \caption{List of algorithms evaluated.}
    \label{tab:abbreviations}
\end{table}

To determine whether a solution $(A,b)$ is a weak CE for an underlying problem, two problems must be solved: first, the underlying problem using these parameters and second, a version of this problem where the solution is enforced to lie within the favorable solution space $\mathcal H$. If the optimal objective value of both problems is equal, the point $(A,b)$ indeed represents a weak CE. Due to the limitations of computational precision, we allow for a tolerance within this definition of $2\cdot 10^{-4}$, i.e. two times Gurobi's default MIP gap. Note that suboptimal solutions for the former problem may result in false positives, while suboptimal solutions for the latter might lead to false negatives. As (WCEP) has been proven to be $\Sigma^p_2$-hard, evaluating a point $(A,b)$ is already NP-hard in general. We therefore report for a selection of the experiments the number of so-called oracle calls. Verifying one $(A,b)$ point as described is counted as one oracle call.

The distance between a weak CE $(A,b)$ and the original instance $(\hat A, \hat b)$ of the underlying optimization problem was measured using the sum of absolute differences (entrywise $L_1$-norm). For all algorithms evaluated, the minimal distance between the closest CE found and the original instance is reported as objective value.

\subsection{Knapsack}\label{sub:knapsack}
\subsubsection{Sample Size and Distribution}\label{sub:knapsack:sampling}
In this section, the impact of the number of samples drawn in each iteration of the dynamic sampling ($N$) is analyzed. For this purpose we use the dynamic sampling step described in Section~\ref{sub:algs:sampling} as a standalone process and measure the share of instances for which a feasible solution for \eqref{eq:weak_CE} was found, as well as the runtime. Since all further steps of the proposed methods depend on the existence of a feasible solution for \eqref{eq:weak_CE}, this step is a crucial part of the proposed algorithms. The results of this experiment therefore support the selection of appropriate parameters for further testing. For each of the 159 Knapsack instances ten runs using different seeds were executed and the results were averaged. The sampling was conducted using uniform (\texttt{ud}) and truncated normal distributions (\texttt{tnd}). The specific parameters of the distributions used are reported in Appendix~\ref{app:params}.

\begin{figure}[h!tb]
    \begin{subfigure}[t]{0.49\textwidth}
         \centering
         \includegraphics[width=1\linewidth]{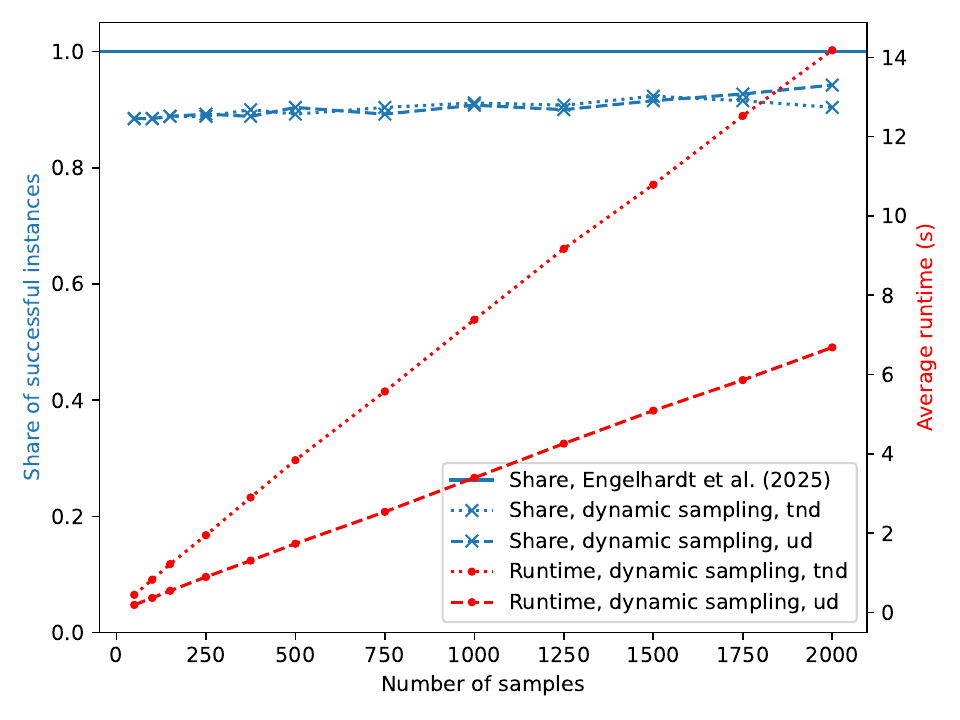}
         \caption{10 items, 26 instances.}
         \label{fig:sample_size:10}
    \end{subfigure}
    \begin{subfigure}[t]{0.49\textwidth}
         \centering
         \includegraphics[width=1\linewidth]{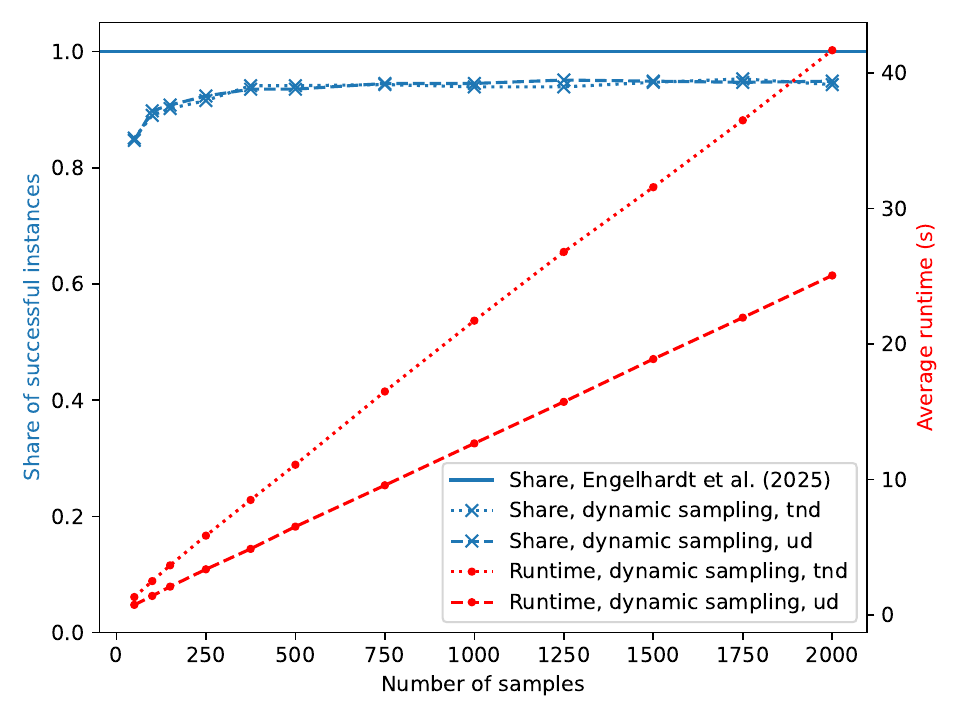}
         \caption{20 items, 51 instances.}
         \label{fig:sample_size:20}
    \end{subfigure}
    \begin{subfigure}[t]{0.49\textwidth}
         \centering
         \includegraphics[width=1\linewidth]{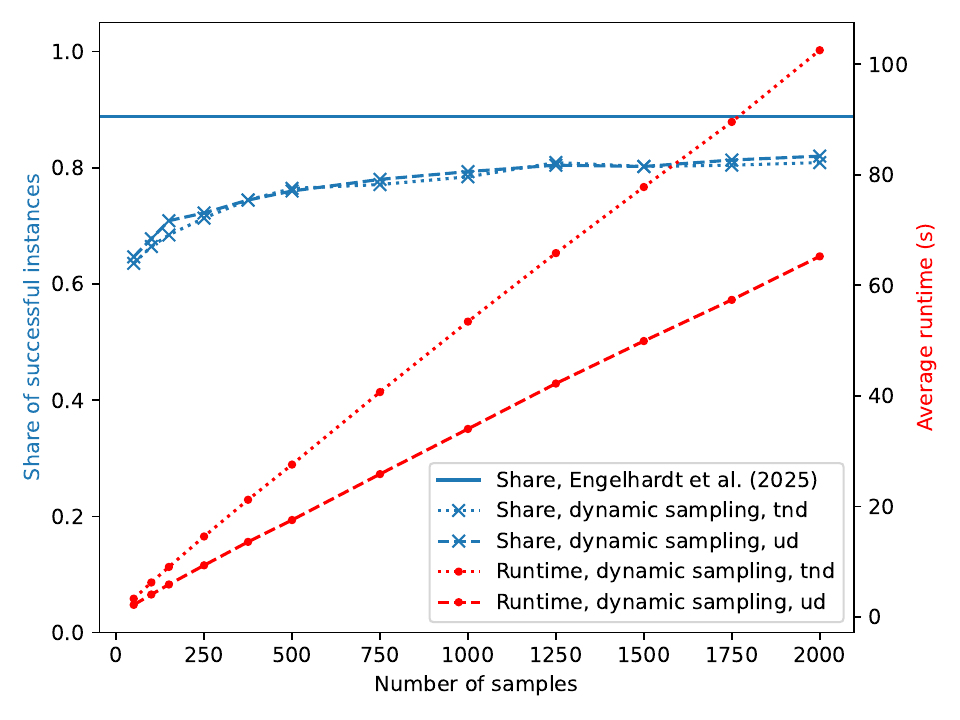}
         \caption{30 items, 45 instances.}
         \label{fig:sample_size:30}
    \end{subfigure}
    \begin{subfigure}[t]{0.49\textwidth}
         \centering
         \includegraphics[width=1\linewidth]{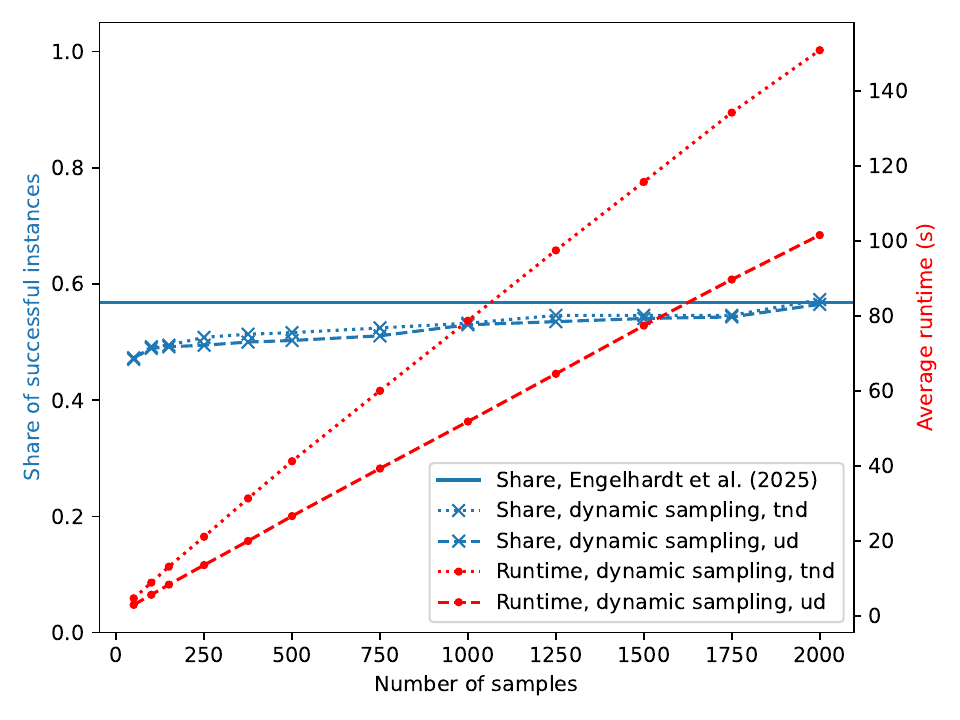}
         \caption{40 items, 37 instances.}
         \label{fig:sample_size:40}
    \end{subfigure}
    \caption{Sample size vs. share of instances for which a weak CE was found and average runtime. In total, all 159 instances are represented.    \label{fig:corr}}
\end{figure}

Figures \ref{fig:sample_size:10}-\ref{fig:sample_size:40} depict the share of instances for which a weak CE was found, and the corresponding average runtime. The solid line indicates the share of instances for which a weak CE was found by \exact. Represented values $<1$ indicate the use of instances for which it is currently unknown if weak CEs within the mutable parameter space exist. The share of instances for which a CE was found by using a uniform distribution and a truncated normal distribution is shown by dashed and dotted lines, respectively. Average runtimes in seconds are reported for the two sampling-based approaches.

It can be seen that there is a positive correlation between the number of samples and the share of CEs found. For all four groups of instances, this relationship holds. Furthermore, larger instances seem to demand a larger sample set. For example, while a sample size of 50 is already sufficient to find weak CEs for 88.6\% of the instances consisting of 10 items, for instances with 30 items a sample size of 750 is necessary to ensure the same share of instances for which the existence of a CE is known. 

The results indicate that both distribution types result in approximately the same share of instances for which a CE was found. However, the use of truncated normal distributions requires up to twice the runtime of the uniform distributions. This indicates the execution of more iterations during the process. Furthermore, there is a clear positive correlation between the instance size, indicated by the number of knapsack items, and the runtime.
For all instances and the maximum sample size of 2,000, dynamic sampling could reach significantly more than 90\% of the success rate of \exact. For the largest instances with 40 items, uniform distributions and a sample size of 2,000, dynamic sampling led to 2.4 CEs on average which were not found by \exact.

The results underline that the dynamic sampling approach, especially when using a uniform distribution, has a success rate for finding CEs which is comparable to that of the exact algorithm. It was also shown that the number of samples drawn ($N$) is a critical parameter and, particularly given the method's relevance for subsequent steps, should be chosen prudently.

\subsubsection{Comparison of the Approaches}\label{sub:knapsack:algs}
In this section we compare the performance of all presented algorithms with the results found by \cite{engelhardt2025counterfactual}. Based on the results presented in the previous section, we make use of a uniform distribution and a sample size given by the number of items times 50. For the methods \samx, \miqp, and \cartsp all pre- and post-processing steps described in Section~\ref{sub:algs:pp} were used. We also use \cartsp without any pre- or post-processing (\cartpure). As a further benchmark, the results for standalone dynamic sampling, which is utilized in all other algorithms, are presented (\texttt{only sampling}).

\begin{table}[h!]
    \centering
    \begin{tabular}{llll}
        Algorithm & Objective value &Runtime (s) & Oracle calls \\
        \hline
        \exact & 36.00             & \phantom{0}989.44              & n/a \\
        \samx  & 37.00 (+2.78\%)   & \phantom{00}21.06 (-97.87\%)    & 2445 \\
        \miqp  & 46.00 (+27.78\%)  & 1638.61 (+65.61\%)  & 2453 \\
        \cart  & 43.02 (+19.51\%)  & \phantom{00}29.89 (-96.98\%)    & 2712 \\
        \cartpure & 57.34 (+59.27\%) & \phantom{000}9.91 (-99.00\%)    & 1522 \\
        \texttt{only sampling} & 77.10 (+114.17\%) & \phantom{000}9.89 (-99.00\%)     & 1501
    \end{tabular}
    \caption{Comparison of all methods used. For each method the median values are presented. Represented: 125/159 instances. Not represented: for 13/159 instances no CE was found by the heuristics but by \exact; for 2/159 instances a CE was found by the heuristics but not by \exact; for 19/159 instances the existence of a CE remains unknown.}
    \label{tab:all_methods_ks}
\end{table}

In Table~\ref{tab:all_methods_ks} the median objective value, runtime and oracle calls are presented. It can be observed that \miqpsp has exceptionally large runtimes, which even exceed those of \exact. Since the calculated CEs are also worse in terms of objective value, this method offers no viable alternative to the exact method. In contrast, both CART-based methods can be executed very fast \cart~benefits from the pre- and post-processing steps. Even \cartpuresp offers a significant improvement in solution quality over \texttt{only sampling} by only requiring marginally more runtime and oracle calls. Apparently, \samxsp results in CEs which are less than 3\% worse than the ones found by \exact. This is particularly noteworthy given that \samxsp is about 50 times faster. Furthermore, this method offers a significant improvement over the starting solutions found by sampling. Overall, this method offers a trade-off between quality of the explanations and runtime which is likely the most favorable for most applications.

\begin{figure}[h!tb]
    \begin{subfigure}[t]{0.49\textwidth}
         \centering
         \includegraphics[width=1\linewidth]{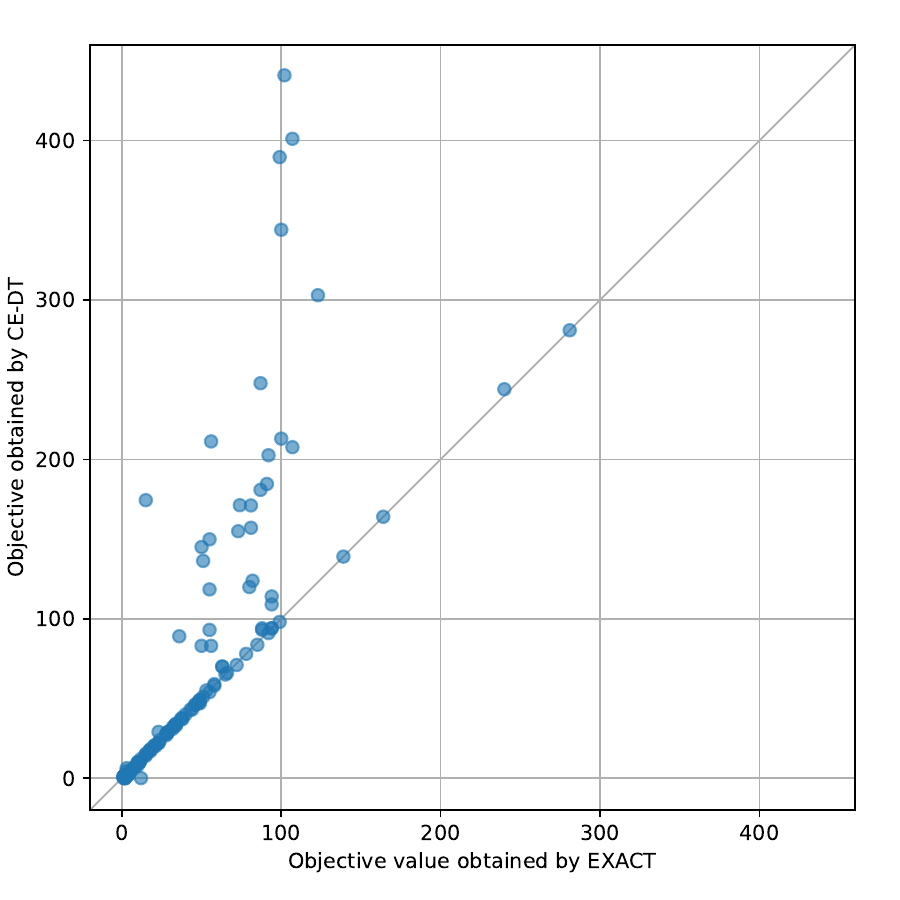}
         \caption{Objective value.}
         \label{fig:engelhardt_vs_samx:obj}
    \end{subfigure}
    \begin{subfigure}[t]{0.49\textwidth}
         \centering
         \includegraphics[width=1\linewidth]{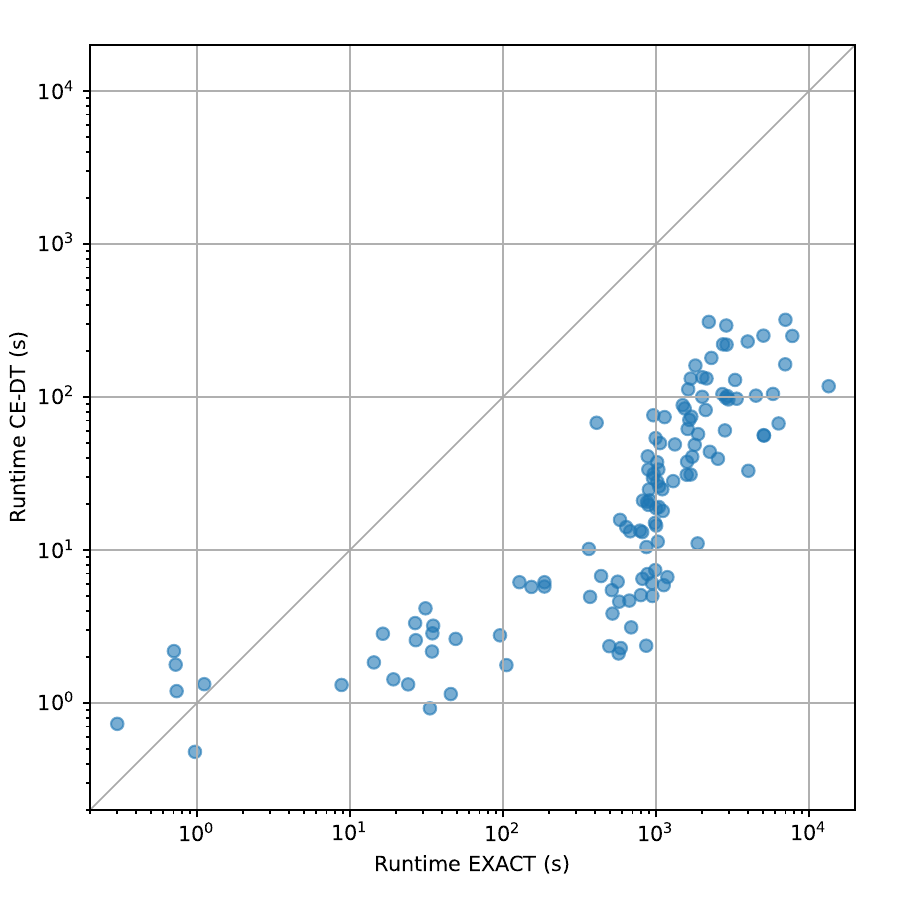}
         \caption{Runtime.}
         \label{fig:engelhardt_vs_samx:time}
    \end{subfigure}
    \caption{\textcolor{black}{Instance-wise comparison of objective value and runtime of the methods \exactsp and \samx.} Represented: 125/159 instances. Not represented: for 13/159 instances no CE was found by the heuristics but by \exact; for 2/159 instances a CE was found by the heuristics but not by \exact; for 19/159 instances the existence of a CE remains unknown. \label{fig:engelhardt_vs_samx}}
\end{figure}

\textcolor{black}{Results for \samxsp and \exactsp are compared in more detail in Figures~\ref{fig:engelhardt_vs_samx:obj} and~\ref{fig:engelhardt_vs_samx:time}. Both figures represent an instance-wise comparison, where every instance for which both methods found a CE is shown as a data point. Figure~\ref{fig:engelhardt_vs_samx:obj} displays the objective values of the best CE found. It can be seen that the heuristic often found solutions of similar quality, especially for instances where the exact approach found a CE close to the original instance. 
In contrast, for instances where the use of \exactsp results in a larger objective value, the relative gap between both methods also becomes larger. It is worth noting that our heuristic found some better solutions than the exact method, which can be explained by two causes: not all instances could be solved to optimality by the exact method and the use of the fractional version of the mutable parameter space allows for potentially better solutions.
Figure~\ref{fig:engelhardt_vs_samx:time} highlights the runtime of both methods. For nearly all instances, the heuristic exhibits significantly faster runtimes. However, for five instances the exact approach was faster than the heuristic. These are cases where the exact algorithm terminated quickly. As a result, the average absolute runtimes for these five instances are only 0.72 seconds and 1.45 seconds, respectively.} 

\subsubsection{Impact of Tree Depth, Pre- and Post-Processing}\label{sub:knapsack:ablation}
As part of this section, the impact of the pre- and post-processing methods presented in Section~\ref{sub:algs:pp} and the maximum depth of the decision tree is analyzed. For this purpose an ablation study, where all combinations of pre- and post-processing methods are tested, was conducted. Furthermore, a range of computational tests with different tree depths was executed. For both studies the method \samxsp was used as previous results revealed the most favorable trade-off between computational effort and solution quality.

\begin{figure}[h!]
    \centering
    \includegraphics[width=0.7\linewidth]{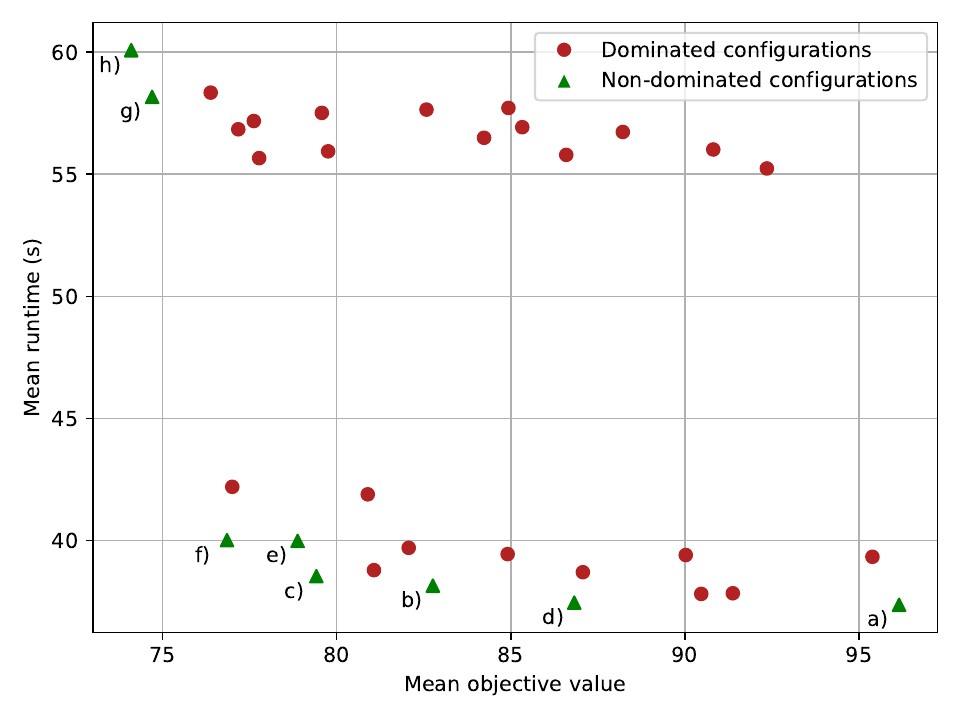}
    \caption{Mean runtime and mean objective value for all parameter combinations tested. Annotations at non-dominated configurations refer to entries of Table~\ref{tab:tab_ablation_small}. Represented: for 127/159 instances a CE was found.}
    \label{fig:ks_ablation_mean}
\end{figure}

\begin{table}[h!]
    \centering
    \begin{tabular}{lllllllll}
    & \multicolumn{8}{c}{Reference points in Figure~\ref{fig:ks_ablation_mean}} \\
    \hhline{~--------}
Method & a) & b) & c) & d) & e) & f) & g) & h) \\
\hline 
\REpre & 0  & 0  & 0  & 0  & 0  & 0  & 1  & 1  \\
\REpost & 0  & 0  & 0  & 0  & 0  & 1  & 1  & 1  \\
\LSone& 0  & 0  & 1  & 1  & 1  & 1  & 1  & 1  \\
\LStwo& 0  & 0  & 0  & 1  & 1  & 0  & 1  & 1  \\
\LSthree & 0  & 1  & 1  & 0  & 1  & 1  & 0  & 1  \\
\hline
Objective& 96.15 & 82.76 & 79.42 & 86.82 & 78.88 & 76.85 & 74.70 & 74.10 \\
Runtime (s)& 37.36 & 38.15 & 38.54 & 37.45 & 39.98 & 40.01 & 58.17 & 60.08 \\
Oracle calls& 1892.68 & 1905.91 & 1931.03 & 1931.73 & 2022.91 & 2056.50 & 2825.55 & 2990.75
    \end{tabular}
    \caption{Non-dominated configurations shown in Fig.~\ref{fig:ks_ablation_mean}. A value of one indicates the use of the respective method. Objective value, time and oracle calls are averaged. Represented: for 127/159 instances a CE was found.}
    \label{tab:tab_ablation_small}
\end{table}

Figure~\ref{fig:ks_ablation_mean} shows the results of the ablation study. We evaluated all possible combinations of the use or non-use of the two resampling and the three line search methods, resulting in the 32 data points shown. Objective value and runtime in seconds are shown as mean over the 127 instances, for which valid CEs were found. We distinguish between dominated (red circles) and non-dominated (green triangles) configurations, the latter lying on the Pareto-front. These configurations and the respective performance indicators are shown in detail in Table~\ref{tab:tab_ablation_small}. Extended results for all configurations are presented in Appendix~\ref{app:ablation}. Two distinct clusters are revealed: The top cluster exclusively contains configurations making use of resampling (\REpre). In contrast, in all runs represented in the bottom cluster this method is not used. This can be explained by the fact that this step not only includes the sampling and evaluation of more data points but also increases the runtime in all subsequent steps, as the set $\mathcal{S}$ is enlarged. Utilizing all pre- and post-processing methods together indeed results in the best CEs (data point h)). However, this setting also results, as expected, in the highest overall runtime. It turns out that between this configuration and the one that involves no additional steps at all (a)), there are various other options that may be favorable, depending on preferences regarding the quality and runtime-trade-off.

\begin{table}[H]
    \centering
    \begin{tabular}{lllll}
        \multirow{2}{1.7cm}{Maximum depth ($d$)} & \multirow{2}{1.2cm}{Actual depth} & \multirow{2}{2.1cm}{Objective value} & \multirow{2}{2.1cm}{Runtime (s)} & \multirow{2}{2.1cm}{Oracle calls}\\ \\
        \hline
        2 & 1.93 & 79.59 & 39.78 & 5657.35 \\
        4 & 3.29 & 76.36 (-4.06\%) & 49.40 (+24.18\%)& 5847.84 (+3.37\%)\\ 
        6 & 4.34 & 74.10 (-6.90\%) & 55.19 (+55.19\%)& 5981.50 (+5.73\%)\\
        8 & 5.09 & 74.16 (-6.82\%) & 60.06 (+60.06\%)& 6243.09 (+10.35\%)\\
        10 & 5.71 & 73.82 (-7.25\%) & 64.03 (+64.03\%)& 6493.07 (+14.77\%)
    \end{tabular}
    \caption{Comparison of the use of different values for the maximum tree depth ($d$). Average values are reported. Represented: for 127/159 instances a CE was found.}
    \label{tab:ks_tree_depth}
\end{table}

Table~\ref{tab:ks_tree_depth} illustrades another trade-off between those two metrics which is implied by the maximum depth of the decision tree trained. Not surprisingly, increasing this parameter leads to constantly higher runtimes. It appears, however, that this increase is only partly attributable to the increasing number of oracle calls. Thus, computing deeper trees takes significantly more time.
In contrast, the improvement in quality of the obtained CEs stagnates for maximum tree depths of six or larger. It can also be seen that for larger values, the actual tree depth achieved does not reach those boundaries.

\subsection{CFLP}\label{sub:cflp}
In this section, we present computational results for the CFLP instances. Based on the results found in the previous experiments, we make use of the methods \samx, \cartpure, and \texttt{only sampling}. For \samx, we furthermore compare the matrix- and the row-based approach. We make use of all pre- and post-processing steps and limit the depth of the decision trees to six.

\begin{figure}[h!]
    \centering
    \includegraphics[width=0.99\linewidth]{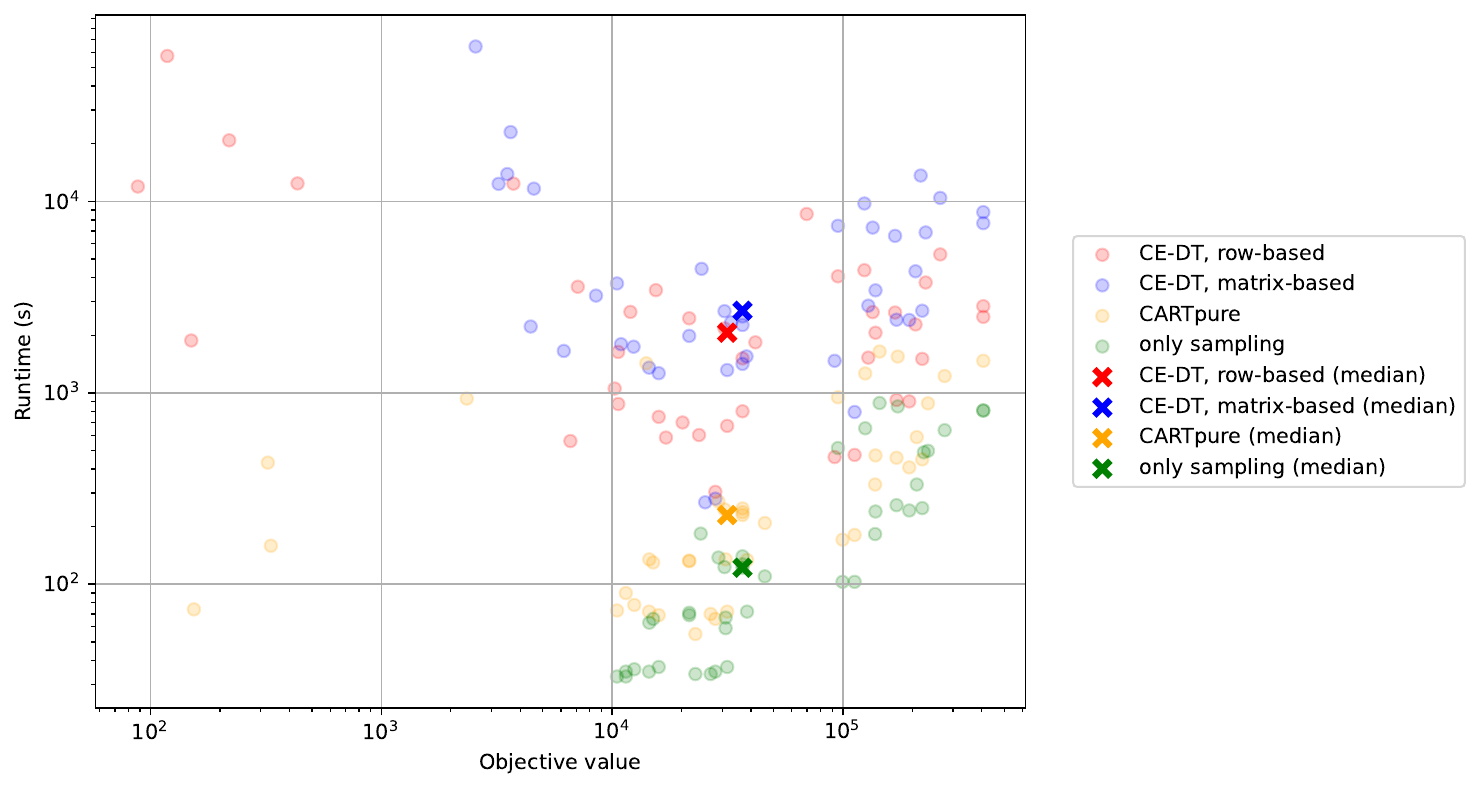}
    \caption{Results for instances with favored solution space of type $\mathcal D^+$. 9/50 instances for which no CE was found are not represented.}
    \label{fig:cflp_d_pos}
\end{figure}

Figures~\ref{fig:cflp_d_pos} and \ref{fig:cflp_d_neg} present runtimes and objective values for the instances for which feasible solutions were found. We distinguish between the two types of favorable solution spaces by using two figures as the results for both differ noticeably. When evaluating the results, it is important to note that there is currently no way to determine actual optimal solutions and their corresponding objective values which could serve as a benchmark. In Figure~\ref{fig:cflp_d_pos} results for using the favored solution space of type $\mathcal D^+$ are shown. It can be seen that all methods achieve CEs of similar quality, with method \cartpuresp slightly outperforming the other methods. The row-based version of \samxsp results in slightly better solutions than the matrix-based version, which are also obtained noticeably faster. Nevertheless, both methods are computationally much more expensive than only using sampling or the \cartpuresp approach.

\begin{figure}[h!]
    \centering
    \includegraphics[width=0.99\linewidth]{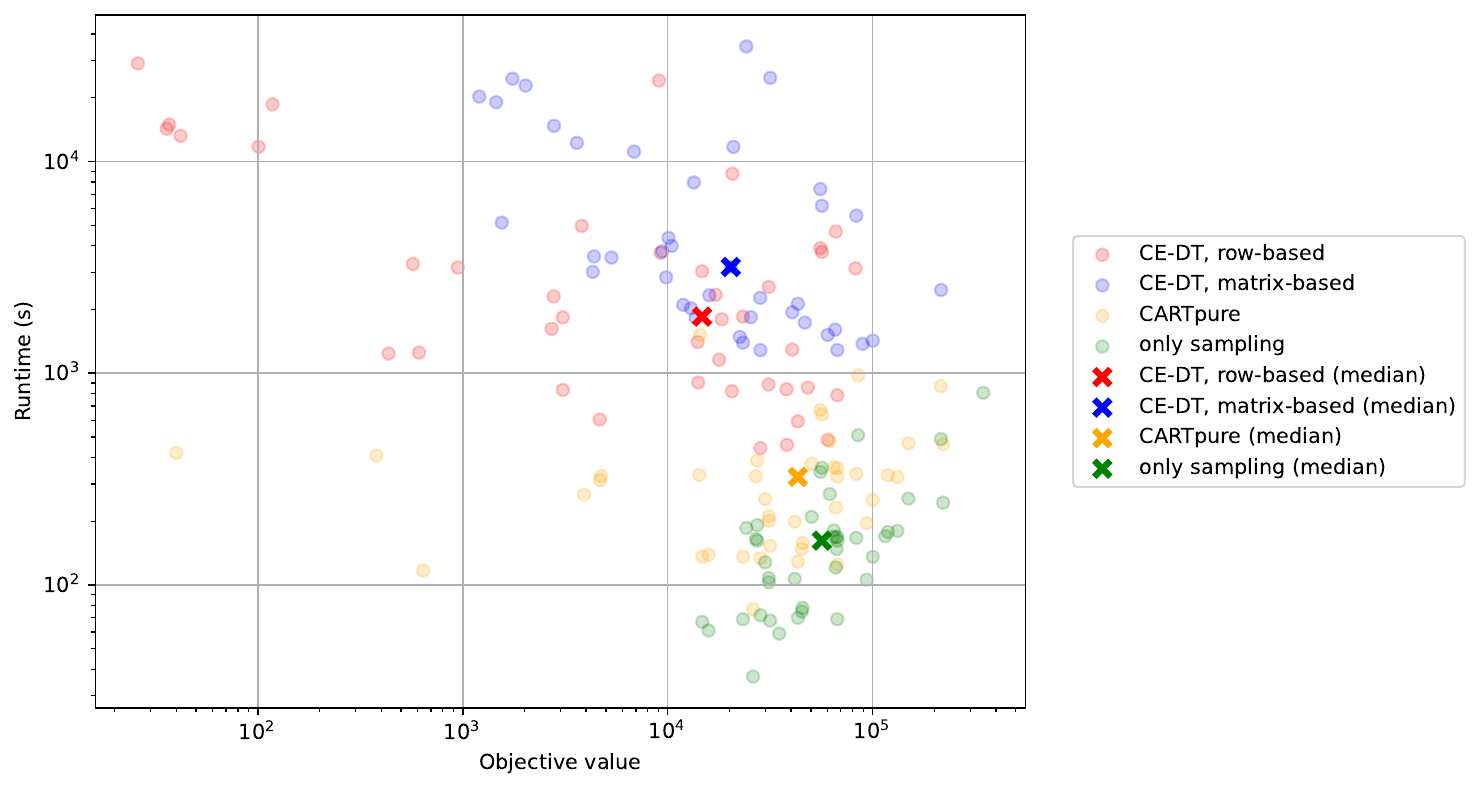}
    \caption{Results for instances with favored solution space of type $\mathcal D^-$. 9/50 instances for which no CE was found are not represented.}
    \label{fig:cflp_d_neg}
\end{figure}

In contrast, for the instances with a favored solution space of type $\mathcal D^-$, presented in Figure~\ref{fig:cflp_d_neg}, both versions of method \samxsp result in solutions which are significantly closer to the original instances than the ones obtained by plain sampling or \cartpure. In general, it can be noted that the improvement in terms of solution quality is much larger for instances considering this type of favorable solution space. In accordance with the previous results, the row-based version of \samxsp clearly outperforms the matrix-based version as it is faster and results in better solutions. 

These results show that \samxsp is also applicable for large instances in which the parameters of multiple constraints are mutable and a subset of the decision variables are continuous. However, based on the structure of the underlying instances, there also appear to be constellations in which the application of the CART-based algorithm seems preferable. 

\section{Conclusion \& Outlook}

Recently, the topic of explainability in mathematical optimization has received increasing attention. This is fueled by the desire to increase trust in decisions proposed by black-box algorithms. Nevertheless, few methods currently exist to address this need. One promising concept is the use of CEs for optimization problems. Although heuristics have been developed for the case of LPs, determining CEs for larger IP instances is still practically impossible.

In this paper, we present a novel approach to heuristically compute CEs for integer optimization problems. We prove that under classical assumptions the feasible region of (\ref{eq:weak_CE}) is essentially a union of polyhedra and that optimal CEs lie on the boundary of these polyhedra which are parametrized by feasible solutions of the underlying present problem. This allows a novel type of decision tree to approximate this region, enabling us to subsequently generate a CE for this tree at relatively low computational costs. We introduce further pre- and post-processing steps which improve the quality of the calculated CEs.
Using Knapsack instances, we benchmark against an exact method for problems with a single mutable constraint and show that we can calculate CEs of similar quality in a fraction of the time. Our ablation study shows that all the component introduced in the pre- and post-processing have an actual effect on the quality of the calculated CE. Of the two version of the proposed \samxsp algorithm, the version querying a single constraint instead of the full constraint matrix was consistently found to perform best.
Furthermore, the results empirically prove the applicability of our method to MIPs, surpassing the scope of our theoretical findings which focus on IPs.

Future work may include the development of methods to enhance the generation of the training set $\mathcal S$. The approach of dynamical sampling apparently has a rate of finding a valid CE in 90\% of the cases where the feasibility of (\ref{eq:weak_CE}) is proven. Even though the development of a computationally efficient method guaranteeing to find a CE seems unlikely due to the proven complexity of (\ref{eq:weak_CE}), we believe that further improvements are possible by, for example, the incorporation of domain knowledge. Furthermore,  query candidates in the \samxsp algorithm are currently found as a byproduct of sampling. Alternative approaches to generate candidates may further increase the performance of this method. Finally, our method is not designed to verify if the weak CE problem is infeasible. Fast methods to check infeasibility are necessary.

\bibliographystyle{alpha}
\bibliography{lit}

\newpage
\appendix

\section{Row-Based Formulation of the \samxsp Algorithm}\label{app:rowbased}
Algorithm~\ref{alg:sample_x_row} describes a variant of the \samxsp method in which a specific constraint of the underlying problem is checked for feasibility. Thus, a query node checks whether the inequality $A_ix\geq b$ is satisfied. In contrast, in the version described in Section~\ref{sub:algs:samx}, we determine queries which check if for an $(A,b)$ sample the system of linear inequalities $Ax\geq b$ holds. 

\begin{algorithm}
\begin{algorithmic}[1]
\Ensure samples $\mathcal S$, tree depth $d$, classification threshold $\beta\in (0,1)$ 
\State $\textcolor{black}{X^{cand} \gets \emptyset,} Q \gets 2^{d-1}-1, \mathcal S^1 \gets\mathcal S , \mathcal L \gets \emptyset$
\ForAll{$(A,b) \in \mathcal S$} \Comment{Generate x candidates}
\State $x^1 \gets \argmin \{ \hat c^\top x: A^i x \ge b^i, x\in\mathcal X\}$ \label{eq:samx_row:oc1}
\State $x^2 \gets \argmin \{ \hat c^\top x: A^i x \ge b^i, x\in\mathcal X\cap \mathcal D\}$ \label{eq:samx_row:oc2}
\State $X^{cand} \gets X^{cand} \cup \{x^1,x^2\}$
\EndFor
\ForAll{$q \in [Q]$} \Comment{Build tree iteratively} \label{eq:samx_row:iter_q}
    \If{$\mathcal S^{q} \neq \emptyset$}
    \State $\mathcal S^{\text{left}} \gets \emptyset, \mathcal S^{\text{right}} \gets \emptyset, g^*\gets0.5, i^* \gets \text{None}$ 
    \State $\bar x^q \gets \text{None}, \bar i^q \gets \text{None},\mathcal S^{\text{leftchild(q)}}\gets \emptyset,\mathcal S^{\text{rightchild(q)}}\gets \emptyset$
        \ForAll{$x \in X^{cand}$} \Comment{Determine best query}\label{eq:samx_row:iter_x}
            \ForAll{$i \in [m]$}\label{eq:samx_row:iter_i}
            \ForAll{$(A,b) \in \mathcal S^{q}$}
            \If{$A^ix \geq b^i$} \label{eq:samx_row:check_i}
                \State $ \mathcal S^{\text{left}} \gets \mathcal S^{\text{left}}\cup (A,b)$
            \Else 
                \State $ \mathcal S^{\text{right}} \gets \mathcal S^{\text{right}} \cup (A,b)$
            \EndIf
            \EndFor
        \If{$\text{ginisplit}(\mathcal S^{\text{left}},\mathcal S^{\text{right}}) < \min\{\text{gini}\text(\mathcal S^{q}),g^*\}$} \label{eq:samx_row:newbestsplit}
            \State $\bar x^q \gets x,\mathcal S^{\text{leftchild(q)}}\gets \mathcal S^{\text{left}}, \mathcal S^{\text{rightchild(q)}}\gets \mathcal S^{\text{right}}$ 
            \State $g^* \gets \text{ginisplit}(\mathcal S^{\text{left}}, \mathcal S^{\text{right}})$
            \State $i^{*} \gets i$
        \EndIf
        \EndFor
        \EndFor
    \If{$S^{\text{leftchild}(q)} = \emptyset$ \textbf} \Comment{Identify leaf nodes}
        \State $\mathcal L \gets \mathcal L \cup\{q\}$
    \ElsIf{$\text{leftchild}(q) \in \{2^{d-1}+1, ...,2^d-1\}$}
    \State $\mathcal L \gets \mathcal L \cup\{\text{leftchild}(q), \text{rightchild}(q)\}$
    \EndIf
    \EndIf
\EndFor
\State label leaves $\ell \in \mathcal L$ as ``feasible'' if $\frac{|\{(A,b)\in \mathcal S^\ell:(A,b)\in \mathcal W\}|}{|\mathcal S^{\ell}|} \geq \beta$ else ``infeasible'' \label{eq:samx_row:labeltree}
\State \Return classification tree on $\mathcal H$ with splits $(\bar x^q, \bar i^q)\ \forall q \in [Q]$ \label{eq:samx_row:return}
\end{algorithmic}
\caption{Greedy algorithm for training the decision tree (matrix based).} 
\label{alg:sample_x_row}
\end{algorithm}

The main structure of Algorithm~\ref{alg:sample_x_row} is identical to Algorithm~\ref{alg:sample_x}, therefore the following description only focuses on the differences between both methods. We now iterate not only over all candidate solutions (Line~\ref{eq:samx_row:iter_x}), but also over all constraints $i \in [m]$ of the underlying problem (Line~\ref{eq:samx_row:iter_i}). In Line~\ref{eq:samx_row:check_i}, we now verify a single constraint with index $i$, instead of a system of inequalities. The queries of the decision tree that is returned (Line~\ref{eq:samx_row:return}) contain not only solution vectors but also the index of the row to be verified.

To compute counterfactuals for all leaves of the decision tree which are labeled as \textit{feasible} we solve the following problem.
\begin{subequations}
\begin{align}
    \min \ &   \| (\hat A, \hat b) - (A, b) \| \\
    s.t. \ & a_i^\top y_F \geq b_i & \forall (y_F,i) \in \mathcal Y_F^\prime & \label{eq:find_ce_for_tree_row:feas} \\
           & a_i^\top y_I \leq b_i - \epsilon & \forall (y_I, i) \in \mathcal Y_I^{\prime} & \label{eq:find_ce_for_tree_row:infeas}\\
           & (A, b) \in \mathcal H
\end{align}\label{eq:find_ce_for_tree_row}
\end{subequations}
We make use of the sets $\mathcal Y_F^{\prime}$ and $\mathcal Y_I^{\prime}$ which contain the solution vectors as well as the corresponding indices of the constraints which must be satisfied (Constraint~\ref{eq:find_ce_for_tree_row:feas}) or violated (Constraint~\ref{eq:find_ce_for_tree_row:infeas}) in order to reach a specific leaf.

\section{MIQP Formulation for Locally Optimal Splits}\label{sub:algs:miqp}
In this section, we present a non-linear mixed-integer programming formulation to determine queries for a decision tree which are locally optimal. In the case, where we require the split to partition some set $\mathcal S = \{(A^{t},b^{t})\}^N_{t=1}$ by checking if any constraint is violated by a vector $x$, the optimal vector $x$ can be obtained by solving Problem~(\ref{eq:gini1}).

As part of this formulation, binary variables $u_t$ indicate whether a data point $t\in [N]$, is feasible for the solution vector $x\in \mathbb Z^n$ by being set to one, or infeasible by being set to zero. This allows us to formulate the minimization of the gini impurity implied by a vector $x$ as the objective function. For the sake of clarity, the objective function (\ref{eq:gini1:obj}) is explicitly formulated in line with the definition in Section~\ref{sec:algs}. Note, however, that it can easily be reformulated to a quadratic form by introducing additional auxiliary variables and quadratic constraints, allowing it to be solved using state-of-the-art solver software. Constraints~(\ref{eq:gini1:l3}) and (\ref{eq:gini1:l4}) enforce binary variables $z_{it}$ to be one, if  constraint $i$ of data point $t$ is violated, otherwise to be zero. Constraints~(\ref{eq:gini1:l5}) only allow $u_t$ to be set to one (which corresponds to $t$ being assigned to the left child node) if all constraints of $t$ are satisfied. Additionally, Constraints~(\ref{eq:gini1:l6}) only allow $u_t$ being set to zero if at least one constraint of $t$ is violated, assigning $t$ to the right child node. 

\begin{subequations}
\begin{align}
    \min \ & \def\arraystretch{2.2}\begin{array}{l} 
        2 \Big( \frac{\sum_{t\in [N]} u_t}{N} \frac{\sum_{t\in [N]} y_t u_t}{\sum_{t\in [N]}u_t} \left(1-\frac{\sum_{t\in [N]} y_t u_t}{\sum_{t\in [N]}u_t}\right) \\ 
        \quad + \frac{\sum_{t\in [N]} (1-u_t)}{N} \frac{\sum_{t\in [N]} y_t (1-u_t)}{\sum_{t\in [N]}(1-u_t)} \left(1-\frac{\sum_{t\in [N]} y_t (1-u_t)}{\sum_{t\in [N]}(1-u_t)}\right) \Big) 
    \end{array} \label{eq:gini1:obj}\\
               s.t. \ & \sum_{j \in [n]} A^t_{ij} x_j \geq b_i^t - M z_{it} & \forall i \in [m], t \in [N] & \label{eq:gini1:l3}\\
               & \sum_{j \in [n]} A^t_{ij} x_j \leq b_i^t + M (1-z_{it}) - \epsilon & \forall i \in [m], t \in [N] & \label{eq:gini1:l4}\\
               & u_t \geq \sum_{i\in [m]} (1-z_{it}) - m +1 & \forall t \in [N] \label{eq:gini1:l5}\\
               & u_t \leq \frac{1}{m} \sum_{i\in [m]} (1-z_{it}) & \forall t \in [N] \label{eq:gini1:l6}\\
               & u_t \in \{ 0,1\} & \forall t \in [N] \label{eq:gini1:dom2}\\
               & x_j \in \Z & \forall j \in [n] \label{eq:gini1:dom3}\\
               & z_{it} \in \{0,1\} & \forall i \in [m], t \in [N] \label{eq:gini1:dom4}
\end{align}\label{eq:gini1}
\end{subequations}

A modification of Problem~(\ref{eq:gini1}) for the row-based case is given by Formulation~(\ref{eq:gini_row}). Constraints~(\ref{eq:gini1:l5}) and (\ref{eq:gini1:l6}) are replaced by (\ref{eq:gini_row:l5}) and (\ref{eq:gini_row:l6}), enforcing that a single constraint $i \in [m]$ of the underlying optimization problem is used for verifying feasibility. This requires the introduction of the binary auxiliary variables $r_i$ (Constraint~(\ref{eq:gini_row:new_var})).

\begin{subequations}
\begin{align}
    \min \ & \def\arraystretch{2.2}\begin{array}{l} 
        2 \Big( \frac{\sum_{t\in [N]} u_t}{N} \frac{\sum_{t\in [N]} y_t u_t}{\sum_{t\in [N]}u_t} \left(1-\frac{\sum_{t\in [N]} y_t u_t}{\sum_{t\in [N]}u_t}\right) \\ 
        \quad + \frac{\sum_{t\in [N]} (1-u_t)}{N} \frac{\sum_{t\in [N]} y_t (1-u_t)}{\sum_{t\in [N]}(1-u_t)} \left(1-\frac{\sum_{t\in [N]} y_t (1-u_t)}{\sum_{t\in [N]}(1-u_t)}\right) \Big) 
    \end{array} \label{eq:gini_row:obj}\\
        s.t. \ & \text{(\ref{eq:gini1:l3}) to (\ref{eq:gini1:l4})} &\tag{10b) to (10c} \\
               & \sum_{i \in [m]} r_i z_{it} =u_t & \forall t \in [N] \tag{10d} \label{eq:gini_row:l5}\\
               & \sum_{i \in [m]} r_i = 1 & \tag{10e} \label{eq:gini_row:l6}\\
               & \text{(\ref{eq:gini1:dom2}) to (\ref{eq:gini1:dom4})} \tag{10f) to (10h} \\
               & r_i \in \{ 0,1\} & \forall i \in [m]  \tag{10l}\label{eq:gini_row:new_var}
\end{align}\label{eq:gini_row}
\end{subequations}

\newpage

\section{Parameters used in the Computational Experiments}\label{app:params}

\paragraph{Parameters of the probability distributions.}
Two different types of probability distributions were used for the experiments presented in this work. Both are parameterized by the scaling parameter $\delta\in \Delta$, where $|\Delta| \in \mathbb Z_+$ represents the iteration limit of the dynamic sampling approach. 

In Sections~\ref{sub:knapsack} and \ref{sub:cflp}, uniform distributions were used. We define $\Delta = \{ 1,\ldots ,K\}$\textcolor{black}{. For a given value of $\delta$ and each mutable parameter $p$ a value is sampled from the interval $[\hat p - (\hat p - p^{lb}) \cdot \frac{\delta}{|\Delta|}, \hat p + (p^{ub}-\hat p)\cdot \frac{\delta}{|\Delta|}]$ with $p^{lb}$ being the lower and $p^{ub}$ the upper bound derived from the mutable parameter space $\mathcal H$. That is, setting $\delta$ to zero would only allow to sample the single point $(\hat A, \hat b)$, whereas setting $\delta :=|\Delta|$ would result in sampling from the full mutable parameter space. This allows us to iteratively enlarge the subset of $\mathcal H$ which is explored in the dynamic sampling process. } 

Furthermore, for a subset of experiments presented in Section~\ref{sub:knapsack:sampling}, truncated normal distributions were used. As part of the dynamic sampling procedure, they are centered at $(\hat A, \hat b)$ and their standard deviations are given by $\sigma = 0.003125 \cdot 2^{\delta-1}$ for $\delta \in \Delta = \{ 1,\ldots ,K\}$. The boundaries of the mutable parameter space $\mathcal H$ are used as cut-off values. In the associated experiments, methods \REpresp and \REpostsp were not used.

\paragraph{Other parameters.}
In Table~\ref{tab:params_appendix}, we list further parameters used in the computational experiments that are not explicitly stated in the respective sections. The entry ``\textit{var}'' indicates that the parameter was changed as part of the experiment and its value is therefore reported in the setup of the experiment. The entry ``-'' highlights that the specific algorithm was not used and hence, no value was set for this parameter.

\begin{table}[h!]
    \centering
    \begin{tabular}{lllll}
        & \multicolumn{4}{c}{Experiments presented in Sec.}\\
        \hhline{~----}
        Parameter & \ref{sub:knapsack:sampling} & \ref{sub:knapsack:algs} & \ref{sub:knapsack:ablation} & \ref{sub:cflp} \\
        \hline
        Algorithm~\ref{alg:dyn_sampling}: $N$ & \textit{var} & 50$I$ & 50$I$ & 100 \\
        Algorithm~\ref{alg:dyn_sampling}: $|\Delta|\textcolor{black}{, K}$ & 5 & 5 & 5 & 5 \\
        Algorithms~\ref{alg:sample_x}, \ref{alg:cart}: $d$ & - & 6 &\textit{var} & 6 \\
        Algorithms~\ref{alg:sample_x}, \ref{alg:cart}: $\beta$ & - & 0.001 & 0.001 & 0.001 \\
        \LStwo: $\alpha$ & - & 0.01& 0.01& 0.01\\
        \REpre: $N_{pre}$ & - & 5 & 5 & 5 \\
        \REpost: $N_{post}$ & - & 5 & 5 & 5 \\
        Problem~\ref{eq:find_ce_for_tree}: $\epsilon$ & - & - & - & 0.00001 \\
        Problems~\ref{eq:find_ce_for_tree}, \ref{eq:gini1}, \ref{eq:gini_row}: $M$ & - & - &- & $2\sum_{i \in [m]}\hat b_i$ \\
        Problems~\ref{eq:gini1}, \ref{eq:gini_row}: $\epsilon$ & - & - & - & 0.001 \\
    \end{tabular}
    \caption{Hyper-parameters used in the computational experiments. ``\textit{var}'' indicates that the respective parameter was subject to changes and the specific values are reported in the respective section. ``-'' indicates that the parameter was not used.}
    \label{tab:params_appendix}
\end{table}

\begin{landscape}
\section{Extended Results of the Ablation Study}\label{app:ablation}
\begin{table}[h!]
    \centering
    \scalebox{.75}{
    \begin{tabular}{lllllllll}
    \multirow{2}{1.9cm}{Reference in Fig.~\ref{fig:ks_ablation_mean}} & \multicolumn{5}{c}{Method} &\multirow{2}{1.6cm}{Objective value}\\
    \hhline{~-----}
        & \REpre & \REpost & \LSone & \LStwo & \LSthree & & Runtime (s) & Oracle calls \\
        \hline 
a) & 0 & 0 & 0 & 0 & 0 & 96.15 & 37.36 & 1892.68 \\
b) & 0 & 0 & 0 & 0 & 1 & 82.76 & 38.15 & 1905.91 \\
-  & 0 & 0 & 0 & 1 & 0 & 90.47 & 37.81 & 1907.09 \\
-  & 0 & 0 & 0 & 1 & 1 & 82.07 & 39.70 & 2011.42 \\
-  & 0 & 0 & 1 & 0 & 0 & 91.37 & 37.84 & 1920.16 \\
c) & 0 & 0 & 1 & 0 & 1 & 79.42 & 38.54 & 1931.03 \\
d) & 0 & 0 & 1 & 1 & 0 & 86.82 & 37.45 & 1931.73 \\
e) & 0 & 0 & 1 & 1 & 1 & 78.88 & 39.98 & 2022.91 \\
-  & 0 & 1 & 0 & 0 & 0 & 95.38 & 39.33 & 1921.93 \\
-  & 0 & 1 & 0 & 0 & 1 & 81.07 & 38.79 & 1957.13 \\
-  & 0 & 1 & 0 & 1 & 0 & 90.02 & 39.40 & 1949.72 \\
-  & 0 & 1 & 0 & 1 & 1 & 80.89 & 41.89 & 2148.07 \\
-  & 0 & 1 & 1 & 0 & 0 & 87.07 & 38.70 & 2026.50 \\
f) & 0 & 1 & 1 & 0 & 1 & 76.85 & 40.01 & 2056.50 \\
-  & 0 & 1 & 1 & 1 & 0 & 84.91 & 39.44 & 2049.80 \\
-  & 0 & 1 & 1 & 1 & 1 & 77.00 & 42.20 & 2225.63 \\
-  & 1 & 0 & 0 & 0 & 0 & 92.35 & 55.24 & 2602.31 \\
-  & 1 & 0 & 0 & 0 & 1 & 79.76 & 55.94 & 2614.45 \\
-  & 1 & 0 & 0 & 1 & 0 & 86.59 & 55.79 & 2616.65 \\
-  & 1 & 0 & 0 & 1 & 1 & 79.57 & 57.51 & 2715.71 \\
-  & 1 & 0 & 1 & 0 & 0 & 88.22 & 56.73 & 2643.65 \\
-  & 1 & 0 & 1 & 0 & 1 & 77.78 & 55.66 & 2654.92 \\
-  & 1 & 0 & 1 & 1 & 0 & 84.23 & 56.49 & 2656.09 \\
-  & 1 & 0 & 1 & 1 & 1 & 77.62 & 57.18 & 2745.79 \\
-  & 1 & 1 & 0 & 0 & 0 & 90.81 & 56.01 & 2643.73 \\
-  & 1 & 1 & 0 & 0 & 1 & 77.18 & 56.84 & 2673.35 \\
-  & 1 & 1 & 0 & 1 & 0 & 85.33 & 56.93 & 2669.48 \\
-  & 1 & 1 & 0 & 1 & 1 & 76.38 & 58.34 & 2849.80 \\
-  & 1 & 1 & 1 & 0 & 0 & 84.93 & 57.72 & 2798.54 \\
g) & 1 & 1 & 1 & 0 & 1 & 74.70 & 58.17 & 2825.55 \\
-  & 1 & 1 & 1 & 1 & 0 & 82.58 & 57.65 & 2821.92 \\
h) & 1 & 1 & 1 & 1 & 1 & 74.10 & 60.08 & 2990.75
    \end{tabular}}
    \caption{Results of the ablation study containing all parameter combinations tested. A value of one indicates the use of
the respective method. Objective value, runtime and oracle calls are averaged.}
    \label{tab:tab_ablation_full}
\end{table}
\end{landscape}

\end{document}